\documentclass[english]{amsart}

\usepackage[english]{babel}
\usepackage{bbm}
\usepackage{hyperref}
\usepackage{verbatim}
\usepackage{ifthen}
\usepackage{amsfonts,amssymb,mathtools}
\usepackage{graphicx,xcolor,subcaption}
\usepackage{enumitem}

\usepackage{pgf}
\newcommand\inputpgf[2]{{
\let\pgfimageWithoutPath\pgfimage
\renewcommand{\pgfimage}[2][]{\pgfimageWithoutPath[##1]{#1/##2}}
\input{#1/#2}
}}

\usepackage[style=numeric,maxcitenames=2,maxbibnames=10,doi=false,isbn=false,url=false,natbib=true,giveninits=true,hyperref,bibencoding=utf8,backend=biber]{biblatex}
\AtEveryBibitem{%
  \clearfield{note}%
  \clearfield{issn}%
  \clearlist{language}%
  }

\usepackage{csquotes}

\newcommand{\N}{\ensuremath{\mathbb{N}}}
\newcommand{\R}{\ensuremath{\mathbb{R}}}

\newcommand{\Z}{\ensuremath{\mathbb{Z}}}
\newcommand{\E}{\ensuremath{\mathbb{E}}}
\renewcommand{\P}{\ensuremath{\mathbb{P}}}

\newcommand{\ind}[1]{\ensuremath{\mathbbm{1}_{\left\{#1\right\}}}}
\newcommand{\diff}{\mathop{}\mathopen{}\mathrm{d}}

\newcommand\steq[1]{\stackrel{\text{\rm #1.}}{=}}

\def\eps{\varepsilon}
\def\cadlag{c\`adl\`ag }

\newcommand{\cal}[1]{\ensuremath{\mathcal{#1}}}

\newtheorem{proposition}{Proposition}
\newtheorem{definition}{Definition}

\newtheorem{theorem}{Theorem}

\title{Stochastic Models of Neural Plasticity:\\ A Scaling Approach}

\date{\today}

\author[Ph. Robert]{Philippe Robert}
\email{Philippe.Robert@inria.fr}
\urladdr{http://www-rocq.inria.fr/who/Philippe.Robert}
\address[Ph.~Robert, G.~Vignoud]{INRIA Paris, 2 rue Simone Iff, 75589 Paris Cedex 12, France}
\author[G. Vignoud]{Ga\"etan Vignoud${}^1$}
\email{Gaetan.Vignoud@inria.fr}
\address[G. Vignoud]{ Center for Interdisciplinary Research in Biology (CIRB) - Coll\`ege de France (CNRS UMR 7241, INSERM U1050), 11 Place Marcelin Berthelot, 75005 Paris, France}
\thanks{${}^1$Supported by PhD grant of \'Ecole Normale Sup\'erieure, ENS-PSL}

\ifpdf
\hypersetup{
  pdftitle={Stochastic Models of Neural Synaptic Plasticity},
  pdfauthor={Philippe Robert and Ga\"etan Vignoud}
}
\fi

\begin{document}

\begin{abstract}
In neuroscience, {\em synaptic plasticity} refers to the set of mechanisms driving the dynamics of neuronal connections, called {\em synapses} and represented by a scalar value, the \emph{synaptic weight}. A Spike-Timing Dependent Plasticity (STDP) rule is a biologically-based model representing the time evolution of the synaptic weight as a functional of the past spiking activity of  adjacent neurons.  A general mathematical framework has been introduced in~\cite{robert_stochastic_2020_1}.

In this paper we develop and investigate a scaling approach  of these models based on several biological assumptions.   Experiments show that long-term synaptic plasticity evolves on a much slower timescale than the cellular mechanisms driving the activity of neuronal cells, like their spiking activity or the concentration of various chemical components created/suppressed by this spiking activity. For this reason, a scaled version of the stochastic model of~\cite{robert_stochastic_2020_1} is  introduced and a limit theorem, an averaging principle, is stated for a large class  of plasticity kernels.  A companion paper~\cite{robert_stochastic_2020} is entirely devoted to the tightness properties used to prove these convergence results.

These averaging principles are used to study  two important STDP models: {\em pair-based rules} and {\em calcium-based rules}. Our results are compared with  the approximations of neuroscience STDP models. A class of discrete models of STDP rules is also investigated for the analytical tractability of its limiting dynamical system.
\end{abstract}

\maketitle

\vspace{-5mm}

\bigskip

\hrule

\vspace{-3mm}

\tableofcontents

\vspace{-1cm}

\hrule

\section{Introduction}
\label{sec:Introduction}
In~\citet{robert_stochastic_2020_1} we have introduced a general class of mathematical models to represent and study synaptic plasticity mechanisms.   Their purpose is to investigate the synaptic weight dynamics, i.e. the evolution of the unilateral connection between two neurons.  These models rely on two clearly stated hypotheses: the effect of plasticity is seen on the synaptic strength on {\em long timescales} and it only depends on the {\em relative timing of the spikes}. This type of plasticity, known as Spike-Timing-Dependent Plasticity (STDP), has been extensively studied in experimental and computational neuroscience, see~\citet{feldman_spike-timing_2012,morrison_phenomenological_2008} for references.

This paper is devoted to a scaling analysis of an important subclass of STDP rules, {\em Markovian plasticity kernels}. These kernels have a representation in terms of finite dimensional vectors whose coordinates are shot-noise processes. See Section~3 of~\cite{robert_stochastic_2020_1} and Section~\ref{MPK-sec} below.

We start with a simple example of the models investigated in this paper. See Section~\ref{sec:scaling} for a detailed presentation.
The stochastic process can be represented by the following variables,
\begin{enumerate}
\item  the membrane potential $X$ of the output cell;
\item the synaptic weight $W$, modeling the strength of the connection from the input neuron to the output neuron.
\end{enumerate}
When the input neuron is spiking (a {\em presynaptic spike}), a chemical/electrical signal is transmitted to the output neuron through the synapse.
If the couple of variables $(X,W)$ is $(x,w)$ just before this event, it is then updated to $(x{+}w,w)$.

In state $X{=}x$, the output neuron emits a spike at rate $\beta(x)$, where $\beta$ is {\em the activation function}. This is a {\em postsynaptic spike}. It is usually assumed that $\beta$ is a nondecreasing function of the membrane potential $X$.

Consequently, after a presynaptic spike and the associated rise in membrane potential $X$, the probability that a postsynaptic spike occurs is increased.  As seen in~\cite{robert_stochastic_2020_1}, STDP synaptic mechanisms depend, in a complex way, on past spiking times of both adjacent neurons.

More formally, in our simple example, the time evolution is described by the following set of stochastic differential equations (SDEs),
\begin{equation}\label{SimpEqW}
\begin{cases}
\diff X(t) &\displaystyle = {-}X(t)\diff t+W(t)\mathcal{N}_{\lambda}(\diff t),\\
\diff Z(t) &\displaystyle =   {-}\gamma Z(t) \diff t+B_1\mathcal{N}_{\lambda}(\diff t)+B_2\mathcal{N}_{\beta,X}(\diff t),\\
\diff W(t) &\displaystyle = Z(t{-})\mathcal{N}_{\beta,X}(\diff t),
\end{cases}
\end{equation}
where $h(t{-})$ is the left-limit of the function $h$ at $t{>}0$ and, for $i{=}\{1,2\}$, $B_i{\in}\R_+$. Throughout the paper, the notation $(Y(t))$ is used to represent the stochastic process $t{\mapsto}Y(t)$ on $\R_+$. 

We discuss briefly the random variables involved.
\begin{enumerate}
\item The point processes $\mathcal{N}_{\lambda}$ and $\mathcal{N}_{\beta,X}$.\\
These random variables are point processes representing the sequences of spike times of the pre- and postsynaptic neurons.

In the present work, ${\cal N}_\lambda$ is assumed to be a Poisson process with rate $\lambda$. It can be represented either as a nondecreasing  sequence  of points $(t_n,n{\ge}1)$ of $\R_+$, or as nonincreasing function, the {\em counting measure}
\[
t{\mapsto}{\cal N}_\lambda((0,t])=\sum_{n\ge 1}\ind{t_n{\le}t}
\]
with jumps of size $1$, or, finally as a random measure, the sum of Dirac measures at the points $(t_n)$,
\[
{\cal N}_\lambda(\diff x) =\sum_{n\ge 1} \delta_{t_n}. 
\]
Each point $t_n$ of ${\cal N}_\lambda(\diff t)$ is associated to a presynaptic spike and the consequent increment of the postsynaptic membrane potential $X(t-)$ by $W(t)$.

The point process $\mathcal{N}_{\beta,X}$ accounts for the instants of the postsynaptic spikes.

It is a nonhomogeneous Poisson process with (random) intensity function $(\beta(X(t{-})))$. See Relation~\eqref{Nbeta} for a formal definition. See~\citet{kingman_poisson_1992} for an extensive introduction to Poisson processes and~\citet{dawson_measure-valued_1993} for theoretical aspects of random measures.
\item The process $(Z(t))$.\\
$(Z(t))$ encodes the past spiking activity of both neurons through an additive functional of $\mathcal{N}_{\lambda}$ and $\mathcal{N}_{\beta,X}$ with an exponential decay factor $\gamma{>}0$.
It is not difficult to see  that, for $t{\ge}0$,
\[
Z(t)=Z(0){+}B_1\int_0^t e^{-\gamma(t-s)}\mathcal{N}_{\lambda}(\diff s){+}B_2\int_0^t e^{-\gamma(t-s)}\mathcal{N}_{\beta,X}(\diff s).
\]
See Lemma~2.1 of~\cite{robert_stochastic_2020_1}. 
In our general model,  $(Z(t))$ is a multidimensional process which can be thought as a vector of cellular processes associated to the  concentration of chemical components created/suppressed by the spiking activity of both neurons. See also~\citet{AndersonKurtz} for a general presentation of stochastic processes in the context of biochemical systems.
\item The processes $(W(t))$. The synaptic weight $W$ is increased at each jump of ${\cal N}_{\beta,X}$ by the value of $(Z(t))$.
\end{enumerate}
From a biological point of view, the relevant process is $(W(t))$, because it describes the synaptic strength, i.e. the intensity of transmission between two connected neurons. This value can be measured through electrophysiological experiments for example. Many computational models have been developed to investigate synaptic plasticity in different contexts.
See~\citet{morrison_phenomenological_2008,graupner_mechanisms_2010,clopath_connectivity_2010,babadi_stability_2016,kempter_hebbian_1999} and the references therein.

From a mathematical perspective, the variables $(X(t),Z(t),W(t))$, solutions of SDE~\eqref{SimpEqW} are central to the model.
However, as will be seen in this article, the point process of instants of postsynaptic spikes ${\cal N}_{\beta,X}$ is the key component of the system since it drives the time evolution of $(Z(t))$ and $(W(t))$ and, consequently, of $(X(t))$.

\subsection{Mathematical Models of Plasticity in the Literature}
\label{subsec:modelsofplasticityliterature}
Numerous works in physics have investigated mathematical models of plasticity.
We quickly review some of them.  Most studies focus on the dynamics of a collection of synaptic weights projecting to a single postsynaptic cell.
There are basically two types of approximations used.
\begin{enumerate}
 \item Separation of timescales.\\
The cellular processes are averaged to give a simpler dynamical system for the evolution of the synaptic weight.
This is a classical approach in the literature. See~\citet{kempter_hebbian_1999,rubin2018natural,eurich1999dynamics,PhysRevE.62.4077}. \citet{Akil2020.04.26.061515} uses an analogous description of the evolution of synaptic weights in the context of a mean-field approximation of several populations of neurons.
This is the approach of the paper, see Section~\ref{secsec:MTS} below.
 \item Fokker-Planck approach.\\
 In this case, the time evolution of the synaptic strength alone is assumed to follow a diffusion process and, consequently, has the Markovian property.
The analysis is done with the associated Fokker-Planck equations and the corresponding equilibrium distribution when it exists.
See \citet{rubin_equilibrium_2001,horn2000distributed,kistler_modeling_2000,van_rossum_stable_2000}.
An extension, the {\em Kramers-Moyal} expansion is also used in this context for some non-Markovian models, see~\citet{10.1162/NECO_a_00267}. We refer to~\citet{pawula1967generalizations, Gardiner} for general properties of the {\em Kramers-Moyal} expansion.
\end{enumerate}
Mathematical studies of models of plasticity are quite scarce.
Most models are centered on evolution equations of neural networks with a fixed synaptic weight. See Sections~1 and~2 of~\citet{robert_dynamics_2016} for a review. In~\citet{abbassian2012neural} and~\citet{perthame2017distributed}, an ODE/PDE approach for a  population of leaky integrate-and-fire neurons  is presented for a specific pair-based STDP rule.  See \citet{chevallier2015microscopic} for the connection between stochastic models and PDE models. \citet{helson_new_2018} investigates a Markovian model of a pair-based  STDP rule. This is one of the few stochastic analyses in this domain.

\subsection*{Multiple Timescales}
\label{secsec:MTS}
An  important feature of long-term neural plasticity explored in this paper is that there are essentially two different timescales in action.

On the one hand, the decay time of the membrane potential and the mean duration between two presynaptic spikes or two postsynaptic spikes are of the order of several milliseconds. See~\citet{gerstner_spiking_2002}.
Consequently, interacting pairs of spikes are on the same timescale.
For example, pair-based models have an exponential decay whose inverse is around 50~milliseconds. See~\citet{bi_synaptic_1998,fino_bidirectional_2005}.  Similarly, for calcium-based models, the calcium concentration decays with a time constant of the order of 20~milliseconds. See~\citet{graupner_calcium-based_2012}. The stochastic process $(Z(t))$ represents fast cellular mechanisms associated to STDP and accordingly, its timescale is also of the same order.

On the other hand, the synaptic weight process $(W(t))$ changes on a much slower timescale. It can take seconds and even minutes to observe an effect of an STDP rule on the synaptic weight. See~\citet{bi_synaptic_1998}.  Computational models of synaptic plasticity have used similar scaling principles.  Kempter et al.~\cite{kempter_hebbian_1999} for  the equation~(1) of this reference for plasticity updates and with different neuronal dynamics, but built in the same framework,  \citet{kistler_modeling_2000}. We can mention also~\citet{van_rossum_stable_2000},  \citet{roberts_computational_1999, PhysRevE.62.4077} where a separation of the timescales is also assumed.  A final example is~\citet{rubin_equilibrium_2001} where the parameter $\lambda$ speeds up the rate of pre- and postsynaptic spikes in the equation for plasticity updates.

Computational models of plasticity incorporate this timescale difference by only implementing small updates of the synaptic weights.
However, it does not really take into account the fact that significant changes occur {\em after} the end of the experiment.
To take into account this phenomenon, a possible approach consists in updating the synaptic weights with a fixed, or
random, delay.
This is not completely satisfactory since the evolution of the synaptic weight is generally believed
to be an integrative process of past events rather than a delayed action. A more thorough analysis is done in SM2 of~\cite{robert_stochastic_2020_1}.
Another approach which we will use consists in implementing this delay through an exponentially filtered process to represent the accumulation of past information.


It is important to stress here, that even if synaptic plasticity depends on the immediate timing of individual spikes, which happens on a fast timescale, it has a slow and delayed impact on the synaptic weight. This justifies the term of long-term plasticity and the fact that we can consider a separation of the timescales.
Fast synaptic plasticity processes also exist, in the sense that they modulate the synaptic weight on the same timescale as the fast neuronal processes (spikes, membrane potential). This is referred to as \emph{short-term} synaptic plasticity. See~\citet{zucker_short-term_2002}. For this type of dynamics, the timescale is of the order of milliseconds,  much faster than the plasticity considered in this paper which can last several hours. This is not investigated in this paper. \citet{Galves_2019} analyzes such models; in this case, separation of the timescales does not occur, and a mean-field approximation is developed.

The scaling approach of this paper represents the model  as a {\em slow-fast} system.
Neuronal processes, associated to the point processes ${\cal N}_{\lambda}$ and ${\cal N}_{\beta,X}$, occur on a timescale which is much faster than the timescale of the evolution of $(W(t))$. For our simple model, with the scaling, the SDE~\eqref{SimpEqW} becomes, for $\eps{>}0$,
\[
\begin{cases}
\diff X_\eps(t) &\displaystyle = {-}X_\eps(t)\diff t/\eps+W_\eps(t)\mathcal{N}_{\lambda/\eps}(\diff t),\\
\diff Z_\eps(t) &\displaystyle =   {-}\gamma Z_\eps(t) \diff t/\eps+B_1\mathcal{N}_{\lambda/\eps}(\diff t)+B_2\mathcal{N}_{\beta/\eps,X_\eps}(\diff t),\\
\diff W_\eps(t) &\displaystyle = Z_\eps(s{-})\eps\mathcal{N}_{\beta/\eps,X_\eps}(\diff t).
\end{cases}
\]
As can be seen, the variables $(X_\eps(t))$ and $(Z_\eps(t))$ evolve on the timescale $t{\mapsto}t/\eps$, with $\eps$ small; they are {\em fast variables}. The increments of the variable $W$ are scaled with the parameter $\eps$, and the integration of the differential element $\eps\mathcal{N}_{\beta/\eps,X_\eps}(\diff s)$ on a bounded time-interval is $O(1)$.  For this reason, $(W_\eps(t))$ is described as a {\em slow process}.
This is a classical assumption in the corresponding models of statistical physics. Approximations of $(W_\eps(t))$ when $\eps$ is small are discussed and investigated with ad-hoc methods. 
The corresponding scaling results, known as separation of timescales, are routinely used in approximations in mathematical models of computational neuroscience; see, for example,~\citet{kempter_hebbian_1999}.

\subsection*{Mathematical Proofs of Averaging Principles}
In a mathematical context, these types of results are referred to as averaging principles.
See~\citet{papanicolalou_martingale_1977} and Chapter~7 of~\citet{freidlin_random_1998} for general presentation.  They have been used to study various biochemical systems, see for example~\citet{BallKurtz} and~\citet{KangKurtz}. See also the general presentation~\citet{Berglund} in the context of dynamical systems and recent developments in~\citet{Kumar}.   We discuss the specific difficulties to prove such convergence results in our stochastic models of STDP rules:

\begin{enumerate}
\item{\sc Tightness of Functionals of Occupation Measures}.\\Recall that the fast process is $(X_\eps(t),Z_\eps(t))$. Part of the technical problems of the proof of an averaging principle is related to the tightness properties of linear functionals of the fast process occupation measures.

  The main difficulty originates, as could be expected, from the scaled point process $\eps\mathcal{N}_{\beta/\eps,X_\eps}(\diff s)$ associated to postsynaptic spikes and, more precisely, from the tightness of families of processes of the form
  \begin{equation}\label{FOC}
\left(\int_0^t Z_\eps(s)\eps\mathcal{N}_{\beta/\eps,X_\eps}(\diff s)\right).
  \end{equation}
This is done in the paper by~\citet{robert_stochastic_2020}. If the model was expressed in terms of functionals of the occupation measure of type
\[
\left(\int_0^t F(X_\eps(s),Z_\eps(s))\diff s\right),
\]
where $s{\to} F(X_\eps(s),Z_\eps(s))$ is a bounded continuous function on $\R_+$, as it is usually the case, the proof of this tightness property would be quite simple.  From this point of view, this is the case of~\citet{BallKurtz},  Kang and Kurtz~\citet{KangKurtz}. In these papers the proof of the tightness results associated to occupation measures is essentially achieved through a quite direct use of ~\citet{karatzas_averaging_1992}. There are technical difficulties, of course, but they are not related to these functionals of occupation measures.  The only (unpublished) paper we know that establishes an averaging principle for a specific pair-based rule of Wilson-Cowan models of neural networks is~\citet{helson_new_2018} and here too, this is a quite direct application of~\citet{karatzas_averaging_1992}.

Due to our quite general framework it does not seem to be possible  to handle functionals of the form~\eqref{FOC} with this approach. The process $(Z_\eps(s))$ is not bounded, and neither is the differential element $\mathcal{N}_{\beta/\eps,X_\eps}(\diff s)$ since $(\beta(X_\eps(s))$ is also not bounded. The proof of this tightness result  motivates a large part of  the most technical estimates of~\citet{robert_stochastic_2020}.

For our general models of~\citet{robert_stochastic_2020} the tightness properties are stated on an a priori, {\em bounded time interval} $[0,S_0)$ only.  More specifically, it is shown that it may happen that the limit in distribution of  $(W_\eps(t_0))$ as $\eps$ goes to $0$ blows up, i.e., hits infinity in finite time $t_0$.  Contrary to all slow-fast results mentioned above where this phenomenon does not occur, convergence is proved on the real half-line. This is an indication perhaps that some stochastic processes  have to be controlled carefully and that the difficulty of the tightness results mentioned above is not an artifact of the method used.

\item {\sc Regularity Properties.}\\
 The results of the paper by~\citet{robert_stochastic_2020}  do not provide convergence results as such. This is the purpose of the present paper of having convergence results and explicit expressions of the asymptotic dynamical system. To have a convergence result as in this paper, regularity properties of the invariant distribution $\Pi_w$ of the fast process $(X_\eps(t),Z_\eps(t))$ when the synaptic weight is fixed at $w$ have to be established. A typical property is that
  \[
  w\longrightarrow \int_{\R_+^{\ell+1}} G(x,z)\pi_w(\diff x, \diff z)
  \]
  is locally Lipschitz for some function $G$ on $\R_+^{\ell+1}$.  This is a delicate question in general, and there are very few cases  where an explicit expression of $\Pi_w$ is known.  This type of result can be proved if there  exists a ``uniform'' Lyapunov function on a neighborhood of $w$, see~\citet{Khasminski}.  Sections~3 and~4 of our paper are devoted to the proof of these type of results. Different arguments are used.
\end{enumerate}

\subsection{Contributions}
\label{subsec:contributions}

A scaled version of Markovian plasticity kernels as introduced in~\citet{robert_stochastic_2020_1} is presented in Section~\ref{sec:scaling}. The difficulty is to take into account the two different timescales mentioned above. This is done by assuming that the membrane potential~$X$ is a ``fast'' variable, i.e.~that it evolves on a fast timescale.  An averaging principle for the synaptic weight process has to be established in this context.

Under convenient assumptions, an averaging principle, Theorem~\ref{th:Stoch}, shows that the evolution equation of the synaptic weight $(W(t))$ converges to a deterministic dynamical system as $\eps$ goes to $0$. The proof of this quite technical result  uses tightness results proved in the companion paper by~\citet{robert_stochastic_2020}.

Sections~\ref{sec:pair} and~\ref{sec:cal} investigate the implications of averaging principles for classical models of pair-based and calcium-based STDP rules. In particular, we work out explicit results for the time evolution of the synaptic weight for several pair-based rules. Related results of the literature in physics are discussed in Section~\ref{comparisoncomp}.

For calcium-based STDP models, the situation is more complicated since an explicit representation of invariant distributions of a class of Markov processes is required to express  the asymptotic time evolution of the synaptic  weight.
Section~\ref{ap:stochqueu} considers an analytically tractable discrete model of calcium-based STDP rules introduced in~\citet{robert_stochastic_2020_1}. With a scaling approach similar to that of Section~\ref{sec:scaling}, the dynamical system verified by the asymptotic synaptic weight can be investigated and an explicit representation of the invariant distributions of the corresponding Markov processes has been obtained in~\citet{robert_stochastic_2020_1}.

\section{A Scaling Approach}
\label{sec:scaling}

We begin with some formal definitions. Two independent point processes are defined on the probability space:
\begin{enumerate}
\item ${\cal N}_\lambda$ is the Poisson process with rate $\lambda{>}0$;
\item  ${\cal P}$ is an homogeneous  Poisson point process on $\R_+^2$ with rate $1$.
\end{enumerate}
If $h$ is a \cadlag function and $(V(t))$ a \cadlag process, we define ${\cal N}_{h,V}$ the point process on $\R_+$ by 
\begin{equation}\label{Nbeta}
\int_{\R_+}f(u){\cal N}_{h,V}(\diff u)\steq{def}  \int_{\R_+^2}f(u)\ind{s{\in}(0,h(V(u-))]}{\cal P}(\diff s,\diff u),
\end{equation}
for any nonnegative Borelian function $f$ on $\R_+$, where ${\cal P}$ is a homogeneous  Poisson point process on $\R_+^2$ with rate $1$. 
The filtration of the space contains the natural filtrations of ${\cal N}_\lambda$ and ${\cal P}$. See~\cite{robert_stochastic_2020_1}. 
\subsection{Markovian Plasticity Kernels}\label{MPK-sec}\  

We go back to the general Markovian formulation of STDP developed in~\cite{robert_stochastic_2020_1}.
Important features are added to the simple model described in the introduction:
\begin{itemize}
\item The dynamics of the membrane potential $X$ is unchanged, except for the influence of a postsynaptic spike on $X$, which is now modeled by a general decrease $x{\rightarrow}g(x){\geq}0$.
\item We consider a multidimensional fast plasticity process $(Z(t))$ in $\R_+^{\ell}$, that can encode the activity of several chemical components.
 They can be defined as {\em shot-noise processes}. A shot-noise process $(S(t))$ associated to a point process ${\cal P}$ on $\R_+$ with amplitude $k(\cdot)$ and exponential decay $\alpha{>}0$ is a solution of the SDE,
  \[
  \diff S(t) = -\alpha S(t) {+}k(S(t{-})){\cal P}(\diff t). 
  \]
See~\citet{gilbert_amplitude_1960} for the corresponding definition. In our case  $(Z(t))$ is a vector of shot-noise processes  associated to $\mathcal{N}_{\lambda}$ and/or $\mathcal{N}_{\beta, X}$, with amplitude $(k_{i}(\cdot))$, $i{=}1$, $2$. 
See Relation~\eqref{eq:markov} below. In this paper,  the term ``shot-noise process'' will refer to the stochastic process defined in this reference, and not to a specific source of neuronal noise, as is usually the case in neurosciences. 
\item The influence of these fast plasticity variables is integrated through general functionals $z{\rightarrow}n_{a,i}(z){\geq}0$ with exponential decay into two slow variables $\Omega_{a}$.
In particular, the process $(\Omega_p(t))$ (resp., $(\Omega_d(t))$) encodes in some way the memory of the spiking activity leading to potentiation, i.e., increasing the synaptic weight (resp., to depression,  i.e., decreasing the synaptic weight).
\item The synaptic weight $W$ is updated thanks to a functional $M$ of both slow plasticity variables and its current value.
\end{itemize}
More rigorously, the random variable $(X(t),Z(t),\Omega_{p}(t),\Omega_{d}(t),W(t))$ is a Markov process, solution of the SDE
\begin{equation}\label{eq:markov}
\begin{cases}
    \quad\diff X(t) &\displaystyle = -X(t)\diff t+W(t)\mathcal{N}_{\lambda}(\diff t)-g\left(X(t-)\right)\mathcal{N}_{\beta,X}\left(\diff t\right),\\
    \quad\diff Z(t) &\displaystyle {=}   (-\gamma\odot Z(t)+ k_0)\diff t\\ &\hspace{2cm}+k_1(Z(t{-}))\mathcal{N}_{\lambda}(\diff t)+k_2(Z(t{-}))\mathcal{N}_{\beta,X}(\diff t),\\
    \quad \diff \Omega_a(t)&\displaystyle ={-}\alpha\Omega_a(t)\diff t {+}n_{a,0}(Z(t))\diff t\\&\hspace{0cm}+
    n_{a,1}(Z(t{-}))\mathcal{N}_{\lambda}(\diff t){+}n_{a,2}(Z(t{-}))\mathcal{N}_{\beta,X}(\diff t),\quad a{\in}\{p,d\},\\
    \quad\diff W(t) &\displaystyle = M\left(\Omega_{p}(t),\Omega_{d}(t), W(t) \right)\diff t.
\end{cases}
\end{equation}
where $(Z(t))$ is a nonnegative $\ell$-dimensional process, $\ell{\ge}1$,  and the following hold:
\begin{itemize}
\item $\gamma{\in}\R_{+}^\ell$, $a{\odot}b{=}(a_i{\times}b_i)$  if $a{=}(a_i)$ and $b{=}(b_i)$ in $\R_+^\ell$.
\item $k_0{\in}\R_+^\ell$ is a constant and  $k_{1}$ and   $k_{2}$  are measurable functions from $\R_+^\ell$ to $\R^\ell$. Furthermore, the $(k_i)$  are such that the function $(z(t))$ has values in $\R_+^{\ell}$ whenever $z(0){\in}\R_+^{\ell}$.
\item For $i{=}0$, $1$, $2$, $n_{a,i}$  is a nonnegative measurable function on $\R_+^\ell$.
\item $M$ is a general measurable function.
\end{itemize}

The firing instants of the output neuron are the jumps of the point process ${\cal N}_{\beta,X}$ on $\R_+$, and the presynaptic spikes are represented by the Poisson process ${\cal N}_\lambda$.

See~\cite{robert_stochastic_2020_1} for more details. Recall the set of assumptions used in this reference.

\medskip
\noindent
{\bf Assumptions~A}
\begin{enumerate}
\item  {\sc Firing Rate Function}.\label{AsF}\\ $\beta$ is a nonnegative, continuous function on $\R$, and $\beta(x){=}0$ for $x{\le}{-}c_\beta{\le}0$.
  \medskip
\item  {\sc Drop of Potential after Firing}.\label{AsD}\\ $g$ is continuous on $\R$ and $0{\le}g(x){\le} \max(c_g,x)$ holds for all $x{\in}\R$, for $c_g{\ge}0$.
  \medskip
\item {\sc Dynamic of Plasticity.}\label{AsM}\\ There exists an interval $K_W{\subset}\R$ such that, for any \cadlag piecewise-continuous functions $h_1$ and $h_2$ on $\R_+$, the ODE
\begin{equation}\label{ODEM}
  \frac{\diff w(t)}{\diff t}{=}M(h_1(t),h_2(t),w(t))
\end{equation}
for all points of continuity of $h_1$ and $h_2$ has a unique continuous solution $(w(t))$ such that $w(t){\in}K_W$ for all $t{\ge}0$ when $w(0){\in}K_W$.
\end{enumerate}

\subsection{A Scaled Model of Markovian Plasticity of Kernels}
\label{secsec:scaledmodelplasticity}
To take into account the multiple timescales mentioned in the introduction, a scaling parameter $\eps{>}0$ is introduced for stochastic processes following ~\eqref{eq:markov}:
\begin{enumerate}
\item The exponential decay of $(X(t))$, $(Z(t))$ and  the rates $\lambda$ and $\beta(\cdot)$ are scaled with the factor~$1/\eps$.
 \item The functions  $n_{a,i}$, $a{\in}\{p,d\}$, $i{\in}\{1,2\}$, associated to synaptic updates due to neuronal spikes are scaled by $\eps$.
\end{enumerate}
The initial condition of $(U_\eps(t))$ is assumed to be fixed:
\begin{equation}\label{InCond}
 U_\eps(0)=U_0 = (x_0,z_0,\omega_{0,p},\omega_{0,d},w_0).
\end{equation}
This leads to the definition of a scaled version of the system~\eqref{eq:markov}, where we denote $(U_\eps(t)){=}(X_\eps(t),Z_\eps(t),\Omega_{\eps,p}(t),\Omega_{\eps,d}(t),W_\eps(t))$:
\begin{equation}\label{eq:markovSAPII}
\begin{cases}
    \diff X_\eps(t) &\displaystyle = {-}\frac{1}{\eps}X_\eps(t)\diff t{+}W_\eps(t)\mathcal{N}_{\lambda/\eps}(\diff t){-}g\left(X_\eps(t-)\right)\mathcal{N}_{\beta/\eps,X_\eps}\left(\diff t\right),\vspace{0.5em}\\
    \diff Z_\eps(t) &\displaystyle =    \frac{1}{\eps}\left(\rule{0mm}{5mm}{-}\gamma\odot Z_\eps(t){+}k_0\right)\diff t\\
    &\hspace{1cm}{+}k_1(Z_\eps(t{-}))\mathcal{N}_{\lambda/\eps}(\diff t){+}k_2(Z_\eps(t{-}))\mathcal{N}_{\beta/\eps,X_\eps}(\diff t),\vspace{0.5em}\\
     \diff \Omega_{\eps,a}(t)&\displaystyle = {-}\alpha\Omega_{\eps,a}(t)\diff t{+}n_{a,0}(Z_\eps(t))\diff t\\
  &\hspace{-5mm}+ \eps \left(\rule{0mm}{5mm} n_{a,1}(Z_\eps(t{-}))\mathcal{N}_{\lambda/\eps}(\diff t){+}n_{a,2}(Z_\eps(t{-}))\mathcal{N}_{\beta/\eps,X_\eps}(\diff t)\right),a{\in}\{p,d\},\vspace{0.5em}\\
    \diff W_\eps(t) &\displaystyle = \rule{0mm}{5mm}M\left(\Omega_{\eps,p}(t),\Omega_{\eps,d}(t), W_\eps(t) \right)\diff t.
\end{cases}
\end{equation}
From Relations~\eqref{eq:markovSAPII}, the dynamics of the processes $(\Omega_{\eps,p}(t))$, $(\Omega_{\eps,d}(t))$ and $(W_\eps(t))$ is slow in the sense that the fluctuations within a bounded time-interval are limited either because of the deterministic differential element $\diff t$ with a locally bounded coefficient, or via a driving Poisson process with rate of order $1/\eps$ but with jumps of size proportional to $\eps$. The processes $(X_\eps(t))$ and $(Z_\eps(t))$ are fast, for each of them the fluctuations are driven either by the deterministic differential element  $\diff t/\eps$, or the jumps of Poisson point processes with rates of the order of $1/\eps$.

\subsection{Averaging Principles}
\label{secsec:homogenizationresults}
We are interested in the limiting behavior of the synaptic weight process $(W_\eps(t))$ when the scaling parameter $\eps$ goes to $0$.
An intuitive, rough picture of the results that can be obtained is as follows: for $\eps$ small enough, on a small time-interval, the slow process $(\Omega_{p,\eps}(t),\Omega_{d,\eps}(t),W_\eps(t))$ is almost constant, and, due to its fast timescale, the process $(X_\eps(t),Z_\eps(t))$  is ``almost'' at its equilibrium distribution $\Pi_{w}$ associated to the current value of $(W_{\eps}(t))$.
If this statement holds in an appropriate way, we can then write a deterministic ODE for the time evolution of a possible limit of  $(\Omega_{p,\eps}(t),\Omega_{d,\eps}(t),W_\eps(t))$.

 We now introduce the framework of our main theorem concerning averaging principles.
If we set the process $(W_\eps(t))$ to be a constant equal to $w$,  the time evolution of $(X_\eps(t),Z_\eps(t))$ in Relation~\eqref{eq:markovSAPII} has the Markov property. The corresponding process will be referred to as the fast process. Its infinitesimal generator is defined as follows:  if $f{\in}\mathcal{C}_b^1\left(\R_+{\times}\R^{\ell}\right)$,  $w{\in}K_W$ and $(x,z){\in}\R{\times}\R_+^{\ell}$, then
\begin{multline}\label{eq:bfast}
B^F_w(f)(x,z)\steq{def}
    -x\frac{\partial f}{\partial x}+\left<-\gamma\odot z{+}k_0, \frac{\partial f}{\partial z}(x,z)\right>\\
    + \lambda\left[\rule{0mm}{5mm}f(x{+}w,z{+}k_1(z)){-}f(x,z)\right]
    + \beta\left(x\right)\left[\rule{0mm}{5mm}f(x{-}g(x),z{+}k_2(z)){-}f(x,z)\right].
\end{multline}

We now introduce a set of general assumptions driving the system~\eqref{eq:markov}.
\medskip

\noindent
{\bf Assumptions~B}
\begin{enumerate}
\item There exists  $C_\beta{\geq}0$ such that
\begin{equation}\label{Condbeta}
  \beta(x) {\leq} C_{\beta}(1{+}|x|)\quad \forall x{\in}\R.
\end{equation}
\item All coordinates of the vector $\gamma$ are positive. There exists $C_k{\geq}0$ such that $0{\le}k_0{\le}C_k$ and the functions $k_{i}$, $i{=}1$, $2$, in ${\cal C}_b^1(\R_+^{\ell},\R_+^{\ell})$, are upper-bounded by $C_k{\geq}0$;
\item There exists a constant $C_n$ such that, for $j{\in}\{0,1,2\}$, $a{\in}\{p,d\}$, the function $n_{a,j}$ is assumed to be nonnegative and Borelian such that $$n_{a,j}(z){\le} C_n(1{+}\|z\|)$$ for $z{\in}\R_+^\ell$, where $\|z\|{=}z_1{+}\cdots{+}z_\ell$.
Additionally, for any $w{\in}K_W$, the discontinuity points of
  \[
  (x,z){\mapsto}(n_{a,0}(z),n_{a,1}(z), \beta(x)n_{a,2}(z))
  \]
  for $a{\in}\{p,d\}$ are negligible for the probability distribution $\Pi_w$ of the operator defined by Relation~\eqref{eq:bfast}.
\item $M$ can be decomposed as, $M(\omega_p,\omega_d,w){=}M_p(\omega_p,w){-}M_d(\omega_d,w){-}\mu w$, where $M_{a}(\omega_a,w)$ is nonnegative continuous function, nondecreasing on the first coordinate for a fixed $w{\in}K_w$, and,
\[
   M_a(\omega_a,w)\leq C_M(1{+}\omega_a),
\]
for all $w{\in}K_W$, for $a{\in}\{p,d\}$.
\end{enumerate}

Note that, in the (large) list of STDP models presented in~\cite{robert_stochastic_2020_1}, only the fast processes of triplet-based and voltage-based models may not verify these assumptions; in particular, the functions $n_{a,i}$ depend on the product of different shot-noises $Z(t)$.
Nevertheless, Assumptions~B result mainly from technical arguments, and is in any way necessary to obtain Theorem~\ref{th:Stoch}. An extension using quadratic functions $n_{a,i}$ instead of linear ones may be proved using the stronger analytical estimations in the proof.

We can now state the main result concerning the scaled model. Its proof is the main result of the paper by~\citet{robert_stochastic_2020}.
\begin{theorem}[Averaging Principle]
\label{th:Stoch}
Under Assumptions~A and~B and for initial conditions satisfying Relation~\eqref{InCond}, there exists $S_0{\in}(0,{+}\infty]$, such that  the family of processes $(\Omega_{p,\eps}(t),\Omega_{d,\eps}(t),W_{\eps}(t),t{<}S_0)$, $\eps{\in}(0,1)$, of the system~\eqref{eq:markovSAPII}, is tight for the convergence in distribution.

As $\eps$ goes to $0$, any  limiting point
$(\omega_p(t),\omega_a(t),w(t),t{<}S_0)$, satisfies the ODEs, for $a{\in}\{p,d\}$,
\begin{equation}\label{ODESimp} 
  \begin{cases}
\displaystyle  \frac{\diff \omega_a(t)}{\diff t}{=}{-}\alpha \omega_a(t){+}\hspace{-1mm}\int_{\R{\times}\R_+^\ell}\hspace{-2mm} \left[\rule{0mm}{4mm}n_{a,0}(z){+}\lambda n_{a,1}(z) {+} \beta(x)n_{a,2}(z)\right]\Pi_{w(t)}(\diff x,\diff z),\\
\displaystyle \frac{\diff w(t)}{\diff t}{=}M(\omega_p(t),\omega_d(t),w(t)),
  \end{cases}
\end{equation}
where, for $w{\in}K_W$, $\Pi_w$ is the unique invariant distribution $\Pi_w$ on $\R{\times}\R_+^\ell$ of the Markovian operator $B_w^F$ defined by Relation~\eqref{eq:bfast}.

If $K_W$ is bounded, then $S_0{=}{+}\infty$ almost surely.
\end{theorem}

\subsection*{Remarks} We quickly discuss several aspects of these results.
\begin{enumerate}
\item {\sc Uniqueness}.\\
If Relation~\eqref{ODESimp}  has a unique solution for a given initial state,  the convergence in distribution of $(W_\eps(t))$  when $\eps$ goes to $0$ is therefore obtained. Such a uniqueness result holds if the integrand, with respect to $s$,  of the right-hand side of Relation~\eqref{ODESimp} is locally Lipschitz as a function of $w(s)$. One therefore has to investigate regularity properties of the invariant distribution $\Pi_w$ as a function of $w$.  This is a quite technical topic; however, methods based on classical results of perturbation theory and their generalizations in a stochastic context (see~\citet{Khasminski}), can be used to prove this type of properties. These problems are investigated for several important examples in Sections~\ref{sec:pair},~\ref{sec:cal}, and~\ref{ap:stochqueu}.
\item {\sc Blow-up Phenomenon.}\\
The convergence properties are stated on a fixed time interval $[0,S_0)$. The reason is that, for some of our models the variable $S_0$ is finite. The limit in distribution of  $(W_\eps(t))$, as $\eps$ goes to $0$, blows up, i.e., hits infinity in finite time. An analog property holds for some mathematical models of large nonplastic random neural networks.  In this case, the blow-up phenomenon is the result of mutually exciting dynamics. In our case, the strengthening of the connection  may grow without bounds when the function $z{\mapsto}n_2(z)$ exhibits some linear growth with respect to $z$ and when the activation function $\beta$ also has a linear growth. See~\citet{robert_stochastic_2020}.
\end{enumerate}

\begin{figure}[ht]
\centerline{\input{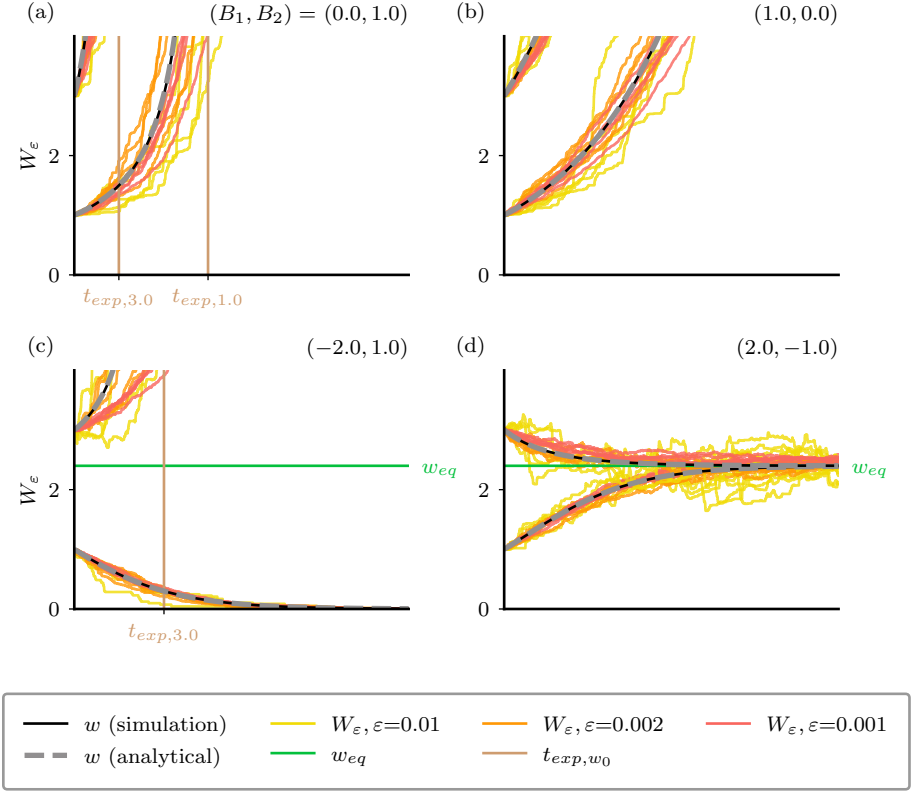}}
\caption[Applications of Stochastic Avergaing to the Toy Model]{Applications of Theorem~\ref{th:Stoch} to the simple model of Section~\ref{sec:Introduction}}
\label{fig:modelpair}
\end{figure}

\subsection*{Simple Model Dynamics}
To illustrate Theorem~\ref{th:Stoch}, we go back to the simple model detailed in Section~\ref{sec:Introduction}.
We present in Figure~\ref{fig:modelpair} different possible behaviors of the asymptotic dynamics, for different values of $B_1$ and $B_2$.
We represent the scaled process for different values of $\eps$ (in yellow, orange and red), the simulated ODE of the asymptotic dynamics (in black) and the analytical solution of the same ODE (in grey, dotted line).
In all cases, the scaled processes converge indeed towards the solution of the ODE.
We simulated the different processes starting from two deterministic initial values $w_0{=}1.0$ and $w_0{=}3.0$.
In particular, we observe different asymptotic regimes:
\begin{enumerate}[label=(\alph*)]
\item The scaled system converges to an ODE that explodes in finite time for all initial conditions. The time of explosion $t_{exp,w_0}$ depends on the initial condition.
\item The limiting ODE leads to solutions of $w(t)$ that diverge towards $+\infty$ but that do not explode.
\item Depending on the initial condition (relative to $w_{eq}$), the asymptotic $w(t)$ either converges to $0$ ($w_0{<}w_{eq}$) or explodes in finite time ($w_0{>}w_{eq}$).
\item The scaled processes converge to an ODE that has a stable fixed point $w_{eq}$, and all the asymptotic processes $w$ converge to this value.
\end{enumerate}
We have shown with this simple example that the blow-up phenomenon does not only result from technical arguments but does indeed take place for some systems.
Moreover, we also highlight the fact that depending on the initial conditions, several behaviors are possible as in Figure~\ref{fig:modelpair}(c).

\medskip
Several important examples of pair-based and calcium-based models are now investigated in light of Theorem~\ref{th:Stoch}.
In order to have simpler expressions, we restrict our study in the following sections to the linear neuron without reset receiving excitatory inputs, leading to the following set of assumptions,

\noindent
{\bf Assumptions~L (Linear)}
\begin{enumerate}
\item The initial conditions of Relation~\eqref{InCond} are such that $U_0{=}(0,0_{\ell},0,0,0)$.
\item The output neuron is without reset; i.e., the function $g$ is null.
The SDE associated to the membrane potential is
\begin{equation}\label{PairW}
\diff X^w(t)={-}X^w(t)\diff t+w\,\mathcal{N}_{\lambda}(\diff t).
  \end{equation}
\item There are only excitatory inputs, i.e., ${0}{\subset}K_{W}{\subset}\R_+$.
\item $M$ verifies Assumptions~A-c and~B-d and is $L_M$-Lipschitz.
\item The activation function is linear, $\beta(x){=}\nu{+}\beta x$, $x{\ge}0$ for $\nu{\ge}0$ and $\beta{>}0$.
\end{enumerate}

In that case, $X$ stays in $\R^+$ and we can have an explicit expression of the stationary distribution of the important point process ${\cal N}_{\beta,X^w}$.
\begin{proposition}\label{PropXZinfty}
Under Assumptions~L, if $(X_\infty^w(t),{-}\infty{\le}t{\le}{+}\infty)$ is a stationary version of SDE~\eqref{PairW}, then the point process ${\cal N}_{\beta,X^w_\infty}$ of Relation~\eqref{Nbeta} extended on the real line is stationary, and if $f$ is a bounded Borelian function with compact support on $\R$, then
\begin{multline}\label{LaplNbeta}
 {-}\ln \E\left[\exp\left(-\int_{-\infty}^{+\infty} f(s){\cal N}_{\beta,X^w_\infty}(\diff s)\right)\right]\\=\nu\hspace{-2mm} \ \int_{-\infty}^{+\infty} \hspace{-2mm} \left(1{-}e^{-f(s)}\right)\diff s
 {+} \lambda \int_{-\infty}^{+\infty} \hspace{-2mm} \left(1{-}\exp\left(-\beta w\hspace{-1mm} \int_0^{+\infty}\hspace{-2mm}  \left(1{-}e^{-f(s{+}t)}\right)e^{-t}\diff t \right)\right)\diff s.
\end{multline}
\end{proposition}
\begin{proof}
Setting
\begin{equation}\label{eqInvX}
  (X_\infty^w(t)){\steq{def}}\left(w\int_{-\infty}^t e^{-(t-s)}{\cal N}_\lambda(\diff s)\right),
\end{equation}
it is easily seen that this process is almost surely defined and that it satisfies Relation~\eqref{PairW}. The stationarity property of $ (X_\infty^w(t))$ and, consequently of  ${\cal N}_{\beta,X^w_\infty}$,  comes from the invariance by translation of the distribution of ${\cal N}_\lambda$.

The independence of ${\cal P}$ and ${\cal N}_\lambda$, and the formula for the Laplace transform of Poisson point processes (see Proposition~1.5 of~\citet{robert_stochastic_2003}) give the relation
\[
\E\left[\exp\left({-}\int_{-\infty}^{+\infty}\hspace{-3mm} f(s){\cal N}_{\beta,X^w_\infty}(\diff s)\right)\right] 
=\E\left[\exp\left({-}\int_{-\infty}^{+\infty} \hspace{-2mm} \left(1{-}e^{-f(s)}\right)\beta(X^w_\infty(s))\diff s\right)\right].
\]
If $F$  is a nonnegative bounded Borelian function with compact support on $\R$, with Relation~\eqref{eqInvX} and Fubini's theorem, we get
\begin{equation}\label{IX}
\int_{-\infty}^{+\infty} F(s) X^w_\infty(s)\diff s
=\int_{-\infty}^{+\infty} \left(w\int_0^{+\infty}F(u{+}s)e^{-s}\diff s\right) {\cal N}_\lambda(\diff u).
\end{equation}
We conclude the proof by using again the formula for the Laplace transform of~${\cal N}_\lambda$
\[
\E\left[\exp\left({-}\int_{-\infty}^{+\infty}\hspace{-3mm} g(s){\cal N}_{\lambda}(\diff u)\right)\right] 
=\E\left[\exp\left({-}\lambda\int_{-\infty}^{+\infty} \hspace{-2mm} \left(1{-}e^{-g(u)}\right)\diff u\right)\right].
\]
where the function $g$ is defined by
\[
g(u)\steq{def} \beta w \int_0^{+\infty}\left(1{-}e^{-f(s{+}u)}\right)e^{-s}\diff s,\quad u{\ge}0.
\]
The proposition is proved.

\end{proof}
\section{Pair-Based Rules}
\label{sec:pair}
We investigate  scaled models of pair-based rules (see~\cite{robert_stochastic_2020}) with Assumptions~L. In this setting,  we are able to derive a closed form expression of the asymptotic equation~\eqref{ODESimp}.

\subsection{All-to-all Model}\label{secsec:sappairalltoall}
We recall the Markovian formulation of the all-to-all pair-based model with exponential functions $\Phi$.
All pairs of pre- and postsynaptic spikes are taken into account in the processes $(\Omega_a(t))$, $a{\in}\{p,d\}$.
See Section~3.1.3 of~\cite{robert_stochastic_2020_1}. In our framework, this is defined as follows. 

\medskip

\noindent
{\bf Assumptions~PA}\\
For $w{\ge}0$, the fast process  associated to the operator $B_w^F$ of Relation~\eqref{eq:bfast} is expressed as $(X^w(t),Z^w(t))$, where $(X^w(t))$ is the  solution of Relation~\eqref{PairW} and
\begin{equation}\label{PairZ}
  \begin{cases}
\diff Z_{a,1}^w(t)={-}\gamma_{a,1} Z_{a,1}^w(t)\diff t+B_{a,1}\,\mathcal{N}_{\lambda}(\diff t),\\
\diff Z_{a,2}^w(t)={-}\gamma_{a,2} Z_{a,2}^w(t)\diff t+B_{a,2}\,\mathcal{N}_{\beta,X^w}(\diff t),
  \end{cases}
\end{equation}
for $a{\in}\{p,d\}$, where $\gamma{=}(\gamma_{a,i}){>}0$ and $B{=}(B_{a,i})$ in $\R_+^4$. For $a{\in}\{p,d\}$ the process $(\Omega_a(t))$ is such that
\[
\diff \Omega_a(t)\displaystyle ={-}\alpha\Omega_a(t)\diff t {+}Z_{a,2}(t{-})\mathcal{N}_{\lambda}(\diff t){+}Z_{a,1}(t{-})\mathcal{N}_{\beta,X}(\diff t).
\]
i.e.,  $n_{a,0}{\equiv}0$, $n_{a,1}(z){=}z_{a,2}$, and $n_{a,2}(z){=}z_{a,1}$ for $z{\in}\R_+^4$.

We denote $\Pi^{\textup{PA}}_w$  the invariant distribution of the process $(X^w(t),Z^w(t))$. The existence of $\Pi^{\textup{PA}}_w$ is given by Proposition~4 of Section~5 of~\citet{robert_stochastic_2020}.

\begin{proposition}\label{PiPAp}
Under Assumptions~L and~PA, then for $a{\in}\{p,d\}$,
\[
\int_{\R{\times}\R_+^4}(\rule{0mm}{4mm}n_{a,0}(z){+}\lambda n_{a,1}(z) {+} \beta(x)n_{a,2}(z))\Pi^{\textup{PA}}_{w}(\diff x,\diff z)
=\frac{\nu}{\beta\lambda}\Lambda_{a,1} {+} \left(\Lambda_{a,1}{+}\Lambda_{a,2}\right)w.
\]
with
\begin{equation}\label{PiPApL}
\Lambda_{a,1}=\beta\lambda^2\left(\frac{B_{a,1}}{\gamma_{a,1}}{+}\frac{B_{a,2}}{\gamma_{a,2}}\right) \text{ and }
\Lambda_{a,2}{=}\beta\lambda \frac{B_{a_1}}{1{+}\gamma_{a,1}}.
\end{equation}
\end{proposition}
\begin{proof}
Assume that the initial point of SDE~\eqref{PairZ} is a random variable $(X^w,Z^w)$ with distribution $\Pi^{\textup{PA}}_w$.

For $a{\in}\{p,d\}$, it is easily seen that $\E\left[Z^w_{a,1}\right]{=}\lambda B_{a,1}/\gamma_{a,1}$ and $\E\left[X^w\right]{=}\lambda w$.
Denote $(Y^w(t)){=}(X^w(t)Z^w_{a,1}(t))$; then with Relation~\eqref{PairZ}, we get
\[
\diff Y^w(t) = {-}(1{+}\gamma_{a,1}) Y^w(t)\diff t+\left(\rule{0mm}{4mm}wZ^w_{a,1}(t{-}){+}B_{a,1} X^w(t{-}){+}wB_{a,1}\right){\cal N}_\lambda(\diff t).
\]
By integrating this ODE on $[0,t]$ and taking the expected value, we obtain
\begin{multline*}
 (1{+}\gamma_{a,1})\E\left[X^wZ^w_{a,1}\right]=
 \lambda w \E\left[Z^w_{a,1}\right]+ \lambda B_{a,1}\E\left[X^w\right] + \lambda w B_{a,1}\\
 = \frac{(1{+}\gamma_{a,1})}{\gamma_{a,1})}\lambda^2 B_{a,1} w + \lambda w B_{a,1}.
\end{multline*}

By integrating the second SDE of Relation~\eqref{PairZ} on $[0,t]$ and taking the expected value, we have
\[
{-}\gamma_{a,2} \E(Z_{a,2}^w)\diff t+B_{a,2}\E(\beta(X^w))=0,
\]
and, with
\[
\E\left[\beta(X^w)\right]=B_{a,2}(\nu {+}\beta\E(X^w))=B_{a,2}(\nu {+}\lambda\beta w),
\]
the proposition is proved. 
\end{proof}
\begin{theorem}
\label{prop:allaffine}
Under Assumptions~L and~PA, as $\eps$ goes to $0$, the family of processes $(\Omega_{\eps,p}(t),\Omega_{\eps,d}(t),W_{\eps}(t))$ of Relation~\eqref{eq:markovSAPII} converges in distribution to the unique solution $(\omega_p(t),\omega_d(t),w(t))$ of the relations
\[
\begin{cases}
    \omega_a(t)&\hspace{-3mm}=\displaystyle \frac{\nu}{\lambda\beta}\Lambda_{a,1}\frac{1{-}e^{-\alpha t}}{\alpha}+\left(\Lambda_{a,1}{+}\Lambda_{a,2}\right) e^{-\alpha t}\int_{0}^te^{\alpha s} w(s) \diff s,\quad a{\in}\{p,d\},\\
 \displaystyle   \frac{\diff w(t)}{\diff t}&\hspace{-3mm}=M\left(\omega_p(t),\omega_d(t), w(t)\right),
\end{cases}
\]
where $\Lambda_{a,i}$, $i{\in}\{1,2\}$, $a{\in}\{p,d\}$ are defined by Relation~\eqref{PiPApL}.
\end{theorem}
\begin{proof}
  This is  a direct consequence of Theorem~\ref{th:Stoch} and of Proposition~\ref{PiPAp}.
\end{proof}
Note  that, for $a{\in}\{p,d\}$, the parameter $\Lambda_{a,1}$  is proportional to the area under the two STDP curves $\Phi_{a,i}(x){=}B_{a,i}\exp({-}\gamma_{a,i}(x))$, $i{=}1$, $2$. It represents the averaged potentiation/depression rate as if we had considered two neurons without any interactions.
Two important facts results from this property,
\begin{itemize}
    \item the term in the dynamics for the constant firing rate of the output neuron, $\nu$, is proportional to $\Lambda_{a,1}$, as expected;
    \item the term $\Lambda_{a,2}$ reflects the dependence between pre- and postsynaptic spikes.
\end{itemize}

\subsection{Nearest Neighbor Symmetric Model}\label{secsec:sappairnearsym}
Similar results can be obtained for the nearest neighbor symmetric scheme of Section~3.1.4 of~\cite{robert_stochastic_2020} with general STDP curves $\Phi$.  For this class of models, whenever one neuron spikes, the synaptic weight is updated by only taking into account the last spike of the other neuron.  In our framework, this is defined as follows. 

\noindent
{\bf Assumptions~PNS}
For $w{\ge}0$, the fast process associated to the operator $B_w^F$ of Relation~\eqref{eq:bfast} can be expressed as $(X^w(t),Z^w(t))$, where $(Z^w(t))$ is the solution of the SDEs,
\begin{equation}\label{PairZ2}
  \begin{cases}
\diff Z_{1}^w(t)= \diff t{-}Z_{1}^w(t{-})\,\mathcal{N}_{\lambda}(\diff t),\\
\diff Z_{2}^w(t)= \diff t{-}Z_{2}^w(t{-})\,\mathcal{N}_{\beta,X^w}(\diff t).
  \end{cases}
\end{equation}
It it easily seen that for  $Z_1^w(t){=}t_0({\cal N}_\lambda,t)$ when $t$ is greater than the first point of ${\cal N}_\lambda$ and, similarly, $Z_2^w(t){=}t_0({\cal N}_{\beta,X^w},t)$ under an analogue condition, with
\begin{equation}\label{t0}
 t_0(m,t){=}t{-}\sup\{s:s{<}t:m(\{s\}){\ne}0\},
\end{equation}
the distance between the first  point of $m$ at the left of $t$ and $t$. For $a{\in}\{p,d\}$, the process $(\Omega_a(t))$ is such that
\[
\diff \Omega_a(t)\displaystyle ={-}\alpha\Omega_a(t)\diff t {+}\Phi_{a,2}(Z_{2}(t{-}))\mathcal{N}_{\lambda}(\diff t){+}\Phi_{a,1}(Z_{1}(t{-}))\mathcal{N}_{\beta,X}(\diff t),
\]
i.e., $n_{a,0}(z){=}0$, $n_{a,1}(z){=} \Phi_{a,2}(z_2)$, and $n_{a,2}(z){=} \Phi_{a,1}(z_1)$ for $z{\in}\R_+^2$.

The functions $\Phi_{a,1}$ and $\Phi_{a,2}$ are quite general nonnegative, nonincreasing, and differentiable functions, instead of exponential functions, as is usually assumed for tractable models of  many STDP rules.

\begin{proposition}\label{PiPaR}
For $w{\ge}0$,  the Markov process $(X^w(t),Z^w(t))$ has a unique invariant distribution  $\Pi^{\textup{PS}}_w$. If $f$ is a bounded Borelian function on $\R_+^2$ and $a{\ge}0$,  then
\begin{align*}
\displaystyle \int_{\R{\times}\R_+^2}f(x,z_1)&\displaystyle\Pi^{\textup{PS}}_{w}(\diff x,\diff z)=\E\left[f\left(we^{-E_\lambda}(1{+}S),E_\lambda\right)\right],\\
\displaystyle\int_{\R{\times}\R_+^2}\ind{z_2{\ge}a}&\displaystyle\Pi^{\textup{PS}}_{w}(\diff x,\diff z) {=}\exp\left(\rule{0mm}{4mm}{-}\nu a
{-}\lambda \hspace{-1mm}\int_{0}^{a} \hspace{-2mm} \left(1{-}\exp\left(-\beta w \left(1{-}e^{s-a}\right)\right)\right)\diff s\right.\\
&\displaystyle\hspace{3cm}\left.\rule{0mm}{4mm}{-}\lambda \int_{-\infty}^{0} \hspace{-2mm} \left(1{-}\exp\left(-\beta w \left(1{-}e^{-a}\right)e^s\right)\right)\diff s\right),
\end{align*}
where $E_\lambda$ and $S$ are independent random variables, $E_\lambda$ has an exponential distribution with rate $\lambda$, and, for $\xi{\ge}0$,
\[
\E\left[e^{-\xi S}\right]=  \exp\left({-}\xi \lambda \int_0^{+\infty} \hspace{-5mm}u  e^{-u} e^{-\xi e^{-u}}\diff u\right).
\]
\end{proposition}
\begin{proof}
The first  condition of Assumption~B-a  is clearly not satisfied, the coordinates of the vector $\gamma$ being ${-}1$. This is not a concern since this condition is only used to construct a Lyapunov function as in the proof of Proposition~4 of Section~5
of~\cite{robert_stochastic_2020}. We only show that one can construct such a function for this model. Set, for $(x,z){\in}\R{\times}\R_+^2$,
\[
H(x,z)\steq{def}\frac{1}{x^\delta}+x{+}z_1{+}z_2,
\]
for some $\delta{>}0$; then,
\[
B_F^w(H)\le \frac{1}{x^\delta}\left( \delta{+}\lambda\left(\frac{x^\delta}{(x{+}w)^\delta}{-}1\right)\right){+}\lambda w + 2 - x -\lambda z_1{-}(\nu{+}\beta x) z_2,
\]

Choosing $\delta{<}\lambda/4$, we set $x_0 = \min\left(x_1,x_2\right)$, where
\[
    x_1=\frac{w}{2^{1/\delta}-1},\quad x_2 = \left(\frac{\lambda }{4(\lambda w{+}3)}\right)^{1/\delta},
\]
such that if $x{\leq}x_0$, then $B_F^w(H)\le -1$.


Moreover, if $x{\geq}x_0$, we also have $B_F^w(H)\le -1$ for $H(x,z){\ge}K_0$, where
\[
K_0=\left(\frac{\delta}{x_0^\delta}{+}\lambda w{+}3\right)/\min(1,\lambda,\nu+\beta x_0).
\]

In particular, $H$ is a Lyapunov function for $B_w^F$. Consequently, there exists a unique invariant distribution.

We denote by $(X^w,Z^w_1,Z^w_2)$ a random variable with distribution $\Pi^{\textup{PS}}_w$. It is easily checked that, for $t{>}0$,
\[
\left(X^w(t),Z^w_1(t)\right){=}\left(w\int_0^t\hspace{-2mm} e^{-(t-s)}{\cal N}_\lambda(\diff s), t_0({\cal N}_\lambda,t)\right){\steq{dist}}
\left(w\int_{-t}^0 \hspace{-2mm}e^{s}{\cal N}_\lambda(\diff s), t_0({\cal N}_\lambda,0)\right),
\]
where $t_0(\cdot,\cdot)$ is defined by Relation~\eqref{t0}. 
By letting $t$ go to infinity, we thus get, with $t_0{\steq{def}} t_0({\cal N}_\lambda,0)$,
\[
(X^w,Z^w_1){\steq{dist}}\left(w\int_{-\infty}^0 \hspace{-3mm} e^{s}{\cal N}_\lambda(\diff s), t_0({\cal N}_\lambda,0)\right)
=\left(we^{-t_0}\left(1{+}\int_{(-\infty,-t_0)}\hspace{-10mm} e^{s{+}t_0}{\cal N}_\lambda(\diff s)\right), t_0\right).
\]
The strong Markov property of ${\cal N}_\lambda$ gives the desired relation for the representation of the law of $(X^w,Z^w_1)$.
Again, with the formula of the Laplace transform of Poisson point processes, we have
\begin{multline*}
  \E\left[\exp\left({-}\xi\int_{(-\infty,-t_0)}\hspace{-10mm} e^{s{+}t_0}{\cal N}_\lambda(\diff s)\right)\right]{=} \E\left[\exp\left({-}\xi \int_{-\infty}^0e^{s}{\cal N}_\lambda(\diff s)\right)\right]\\
  =\exp\left({-}\lambda\int_0^{+\infty} \left(1{-}e^{-\xi e^{-s}}\right)\diff s\right).
\end{multline*}
The stationary distribution of $(Z_2(t))$ is the distribution of the distance of the first point  of ${\cal N}_{\beta,X}$ at the left of $0$ at equilibrium, and hence,  for $a{\ge}0$,
\[
\P\left(Z_2^w{\ge}a\right){=}\P\left({\cal N}_{\beta,X^w_\infty}((-a,0)){=}0\right){=}\P\left({\cal N}_{\beta,X^w_\infty}((0,a)){=}0\right).
\]
Relation~\eqref{LaplNbeta} gives, for $\xi{\ge}0$,
\begin{multline*}
  {-}\ln \E\left[e^{{-}\xi{\cal N}_{\beta,X^w_\infty}((0,a))}\right]
  {=}\nu a\left(\!1{-}e^{-\xi}\!\right)
{+}\lambda\!\! \int_{0}^{a} \hspace{-2mm} \left(1{-}\exp\left(-\beta w \left(1{-}e^{{-}\xi}\right)\!\!\left(1{-}e^{s-a}\right)\right)\right)\!\diff s\\
{+}\lambda \int_{-\infty}^{0} \hspace{-2mm} \left(1{-}\exp\left(-\beta w \left(1{-}e^{-\xi}\right)\left(1{-}e^{-a}\right)e^s\right)\right)\diff s.
\end{multline*}
By letting $\xi$ go to infinity, we have obtained the desired expression. The proposition is proved.

\end{proof}
\begin{theorem}[Averaging Principle] \label{th:PairStochNS}
Under Assumptions~L and~PNS, as $\eps$ goes to $0$, the family of processes $(\Omega_{\eps,p}(t),\Omega_{\eps,d}(t),W_{\eps}(t))$ of Relation~\eqref{eq:markovSAPII}  converges in distribution to $(\omega_p(t),\omega_d(t),w(t))$, the unique solution of the ODE
\[
\begin{cases}
    \omega_a(t)&\hspace{-3mm}=\displaystyle \int_{0}^te^{-\alpha(t-s)}\int_{\R_+^5}\left( \beta(x)\Phi_{a,1}(z_1){+}\lambda \Phi_{a,2}(z_{2}) \right)\Pi^{\textup{PS}}_{w(s)}(\diff x,\diff z) \diff s,\quad a{\in}\{p,d\},\\
 \displaystyle   \frac{\diff w(t)}{\diff t}&\hspace{-3mm}=M\left(\omega_p(t),\omega_d(t), w(t)\right),
\end{cases}
\]
where $\Pi^{\textup{PS}}_{w}$ is defined in Proposition~\ref{PiPaR}.
\end{theorem}
\begin{proof}
For $w{\ge}0$,  let $(X_\infty^w,Z_{\infty,1}^w,Z_{\infty,2}^w)$ be random variables with distribution $\Pi_w$, and, for $a{\in}\{p,d\}$, let
\[
\Psi_a(w)\steq{def}\E\left[\rule{0mm}{4mm}\beta(X_\infty^w)\Phi_{a,1}(Z_{\infty,1}^{w})\right]{+}\lambda \E\left[\rule{0mm}{4mm}\Phi_{a,2}(Z_{\infty,2}^{w})\right].
\]
The ODE can be rewritten as
\[
\frac{\diff w(t)}{\diff t} = M\left(\int_{0}^te^{-\alpha(t-s)}\Psi_p(w(s)) \diff s, \int_{0}^te^{-\alpha(t-s)}\Psi_d(w(s))\diff s, w(t)\right).
\]
With Theorem~\ref{th:Stoch}, all we have to prove is that this ODE has a unique solution. This is a simple consequence of the Lipschitz property of $\Psi_a$. Indeed, first, the distribution of $Z_{\infty,1}^{w}$ does not depend on $w$ and $\beta(\cdot)$ is an affine function of $X_\infty^w$ given by Relation~\eqref{eqInvX}. Finally, the identity
\[
\E\left[\rule{0mm}{4mm}\Phi_{a,2}(Z_{\infty,2}^{w})\right]={-}\int_0^{+\infty}\dot{\Phi}_{a,2}(u)\P(Z_{\infty,2}^{w}{\le}u)\diff u,
\]
Proposition~\ref{PiPaR}, and simple estimations give that the function
$\Psi_a$ has the  Lipschitz property.
The theorem is proved.
\end{proof}

\subsection{Links with Models of Physics}\label{comparisoncomp}

In this section, averaging principles for STDP rules of~\citet{kempter_hebbian_1999} are discussed.
We start by characterizing which type of STDP rules are used, in particular, their model takes into account all pairing of pre- and postsynaptic spikes that last less than the interval of the experiment $T$. It is supposed that $T$ is really large compared to the neuronal dynamics.
Accordingly, in the limit of large $T$, it corresponds to the pair-based all-to-all model of Assumptions~PA.

After adapting notation, the main equation for the asymptotic behavior of the synaptic weight dynamics (Relation~(4) of this reference) is expressed, via a separation of timescale argument, as
\begin{equation}\label{eqKemp}
 \frac{\diff \widetilde{w}}{\diff t}=w^{1}\nu^{1}(t)+w^{2}\nu^{2}(t)+\int_{-\infty}^{+\infty}\widetilde{\Phi}(s)\widetilde{\mu}(s,t) \diff s,
\end{equation}
where $S^1$ (resp., $S^2$) is the process of  presynaptic spikes (resp., postsynaptic spikes),
\begin{itemize}
\item $\nu^{1}(t)=\overline{\langle S^{1}(t)\rangle}$, the presynaptic spike rate and $w^{1}$ the intensity of synaptic plasticity triggered by presynaptic spikes only;
\item $\nu^{2}(t)=\overline{\langle S^{2}(t)\rangle}$, the postsynaptic spike rate and $w^{2}$ the intensity of synaptic plasticity triggered by postsynaptic spikes only;
\item $\widetilde{\Phi}(t)$ represents the STDP curve;
\item $\widetilde{\mu}(s,t){=}\overline{\langle S^{1}(t{+}s)S^{2}(t)\rangle}$, the correlation between the spike trains.
\end{itemize}

The quantity $\overline{\langle{\cdots}\rangle>}$ is defined in terms of {\em temporal and ensemble averages} that are not  completely clear from a  mathematical point of view,  $\langle{\cdots}\rangle$ is the {ensemble average} and $\overline{\cdots}$ is the {temporal average} over the spike trains.
The model of~\citet{kempter_hebbian_1999} is without exponential filtering; see Section~\ref{apap:noexp}.

In our setting, we choose $\overline{M}(\Gamma_p,\Gamma_d,w){=}\Gamma_p{-}\Gamma_d$, $n_{a,0}(z){=}0$, $n_{a,1}(z){=}D_{a,1}{+}z_{a,2}$ and
$n_{a,2}(z){=}D_{a,2}{+}z_{a,1}$, where $z_{a,i}$ are defined as in Assumptions~PA. Theorem~\ref{th:StochNoExp} gives the following equation:
\begin{multline}\label{eqModelNoexpKemp}
 \frac{\diff \overline{w}}{\diff t}=(D_{p,1}{-}D_{d,1})\lambda+(D_{p,2}{-}D_{d,2})\int_{\R_+}\beta(x)\Pi^{\textup{PA}}_{\overline{w}(t)}(\diff x)\\
    +\int_{\R_+}\left(\lambda z_2 {+} \beta(x)z_1\right)\Pi^{\textup{PA}}_{\overline{w}(t)}(\diff x, \diff z_1, \diff z_2),
\end{multline}
where $\Pi_w^{\textup{PA}}$ is defined in Proposition~\ref{PiPAp}.

We then have the following equivalence:
\begin{center}
\begin{tabular}{ |c|c|c| }
 \hline
  & \citet{kempter_hebbian_1999} & Our model  \\
 \hline
 Presynaptic plasticity  & $w^{1}$ & $D_{p,1}{-}D_{d,1}$ \\
 Presynaptic rate &$\nu^{1}(t) $& $\lambda$ \\
 Postsynaptic plasticity & $w^{2}$ & $D_{p,2}{-}D_{d,2}$ \\
 Postsynaptic rate &$\nu^{2}(t)$ & $\displaystyle\int_{\R_+}\beta(x)\Pi^{\textup{PA}}_{\overline{w}(t)}(\diff x)$ \\
 STDP &  $\displaystyle\int_{-\infty}^{+\infty}\widetilde{\Phi}(s)\widetilde{\mu}(s,t) \diff s$ &  $\displaystyle\int_{\R_+}\left(\lambda z_2 {+} \beta(x)z_1\right)\Pi^{\textup{PA}}_{\overline{w}(t)}(\diff x, \diff z_1, \diff z_2)$\\
 \hline
\end{tabular}
\end{center}
\vspace{1em}
The equivalence of the last row can be explained as follows.

We set
\[
    \Phi_{a}(t)=B_{a,1}\exp(-\gamma_{a,1}t)\ind{t>0} + B_{a,2}\exp(\gamma_{a,2}t)\ind{t<0},
 \]
and
\[
\overline{\mu}(t, w)=
\begin{cases}
\displaystyle\;\;  \lim_{h{\searrow}0}\frac{\mathbb{E}_{\Pi^{\textup{PA}}_w}\left(\mathcal{N}_{\lambda}[0,h] \mathcal{N}_{\beta,X}[t, t{+}h]\right) }{h^2}, \quad\text{ for }   t{>}0;\\[10pt]
\displaystyle\;\;  \lim_{h{\searrow}0}\frac{\mathbb{E}_{\Pi^{\textup{PA}}_w}\left(\mathcal{N}_{\lambda}[0,h] \mathcal{N}_{\beta,X}[t,t{+} h]\right) }{h^2}, \quad\text{ for }   t{<}0,
\end{cases}
\]
provided that the limits related to second order properties of the point processes ${\cal N}_\lambda$ and ${\cal N}_{\beta,X}$ exist.

In Section~\ref{app:comparisoncompheur}, a heuristic argument shows that
\[
    \int_{\R_+}\left(\lambda z_2{+}\beta(x)z_1\right)\Pi^{\textup{PA}}_{\overline{w}(t)}(\diff x, \diff z_1, \diff z_2)=\int_{-\infty}^{+\infty}(\Phi_p(s){-}\Phi_d(s))\overline{\mu}(s,\overline{w}(t)) \diff s,
\]
leading to the equivalence between both models.
\section{Calcium-Based Rules}
\label{sec:cal}\label{cbmsec}
We investigate scaled models of calcium-based rules introduced in Section~3.1.1 of~\cite{robert_stochastic_2020_1}. In this section, we show that the asymptotic equation~\eqref{ODESimp} has a unique solution. Some regularity properties of the invariant distribution of the operator $B_w^F$, with respect to the variable $w$,   have to be obtained.

\medskip

\noindent
{\bf Assumptions~C}
In this case, the vector $(Z(t))$ is a nonnegative one-dimensional process $(C(t))$.
For $w{\in}K_W$, the fast process  associated to the operator $B_w^F$ of Relation~\eqref{eq:bfast} can be expressed as   $(X^w(t),C^w(t))$, where, as before, $(X^w(t))$ is the solution of Relation~\eqref{PairW} and the SDE for $(C^w(t))$ is
\begin{equation} \label{eq:CaSysF}
\diff C^w= {-}\gamma C^w(t) \diff t+C_{1}\mathcal{N}_{\lambda}(\diff t)+C_{2}\mathcal{N}_{\beta,X^w}(\diff t),
\end{equation}
where $C_1$ and $C_2{\geq}0$, $\gamma{>}0$. For $a{\in}\{p,d\}$, the process $(\Omega_a(t))$ is such that
\[
\diff \Omega_{a}(t) =\left(-\alpha\Omega_{a}(t) {+} h_a(C(t))\right)\diff t,
\]
i.e.,  $n_{a,0}(c){=}h_a(c)$, $n_{a,1}(c){=}0$, and  $n_{a,2}(c){=}0$ for $c{\in}\R_+$. The functions $h_p$ and $h_d$ are assumed to be  $L$-Lipschitz.
They represent, respectively, the influence of the calcium concentration $C$ on potentiation and depression.

\begin{proposition}\label{def:calciumstati}
  For $w{\in}K_W$, the Markov process $(X^w(t),C^w(t))$ has a unique invariant distribution $\Pi_w^{\textup{C}}$, and its Laplace transform is given by, for $a$ and $b{\ge}0$,
\begin{multline*}
    {-}\ln \int_{\R_+^2}e^{-a x -bc}\Pi_w^{\textup{C}}(\diff x,\diff c)=
    \nu\int_{-\infty}^0\hspace{-2mm} \left(1{-}e^{{-}bC_2 e^{\gamma u}} \right)\diff u\\
           {+}\lambda \int_{-\infty}^0\left(1{-}\exp\left({-}aw e^{u}{-}bC_1e^{\gamma u}-\beta w\int_0^{u}\left(1{-}e^{{-}bC_2 e^{\gamma(u{-}s)}} \right)e^{s}\diff s \right)\right)\diff u.
\end{multline*}
\end{proposition}

\begin{proof}
The existence and uniqueness of $\Pi_w^{\textup{C}}$ is a direct consequence of Theorem~\ref{th:Stoch} since Assumptions~B hold in this case and  Proposition~4 of Section~5 of~\cite{robert_stochastic_2020} can be used.

With  Proposition~\ref{PropXZinfty} and Lemma~2.1 of~\citet{robert_stochastic_2020_1}, a  stationary version $(X^w_\infty(t),C^w_\infty(t))$ of the fast process  $(X^w(t),C^w(t))$ can be represented as
\begin{equation}\label{RepCw}
\left(w\int_{-\infty}^t e^{-(t-s)}{\cal N}_\lambda(\diff s), C_1\int_{-\infty}^t\hspace{-3mm} e^{-\gamma(t-s)}{\cal N}_\lambda(\diff s){+}C_2\int_{-\infty}^t\hspace{-3mm}  e^{-\gamma(t-s)}{\cal N}_{\beta,X_\infty^w}(\diff s)\right).
\end{equation}
Hence, we have to calculate $\E\left[\exp({-}aX_\infty^w(0){-}bC_\infty^w(0))\right]$, that is,
\[
\Psi(a,b)\steq{def}\E\left[\exp\left({-}\int_{-\infty}^0\left(aw e^{s}{+}bC_1e^{\gamma s}\right){\cal N}_\lambda(\diff s){-}bC_2\int_{-\infty}^0 e^{\gamma s}{\cal N}_{\beta,X_\infty^w}(\diff s)\right)\right].
\]
We proceed as in the proof of Proposition~\ref{PropXZinfty}. By the independence of ${\cal P}$ and ${\cal N}_\lambda$,
\[
\E\left.\left[\exp\left({-}bC_2\int_{-\infty}^0\hspace{-4mm} e^{\gamma s}{\cal N}_{\beta,X_\infty^w}(\diff s)\right)\right| {\cal N}_\lambda\right]
{=}\exp\left({-}\int_{-\infty}^0\hspace{-3mm}\left(1{-}e^{{-}bC_2e^{\gamma s}}\right) \beta(X_\infty^w(s))\diff s\right)
\]
and, with the help of Relation~\eqref{IX}, we follow the same methods to obtain the desired result.

\end{proof}

\begin{theorem}\label{th:CaStoch}
Under Assumptions~L and~C, if the functions $h_p$ and $h_d$ are Lipschitz then, as $\eps$ goes to $0$, the family of processes $(\Omega_{\eps,p}(t),\Omega_{\eps,d}(t),W_{\eps}(t))$ converges in distribution to the unique solution $(\omega_p(t),\omega_d(t),w(t))$ of the relations
\begin{multline}\label{eqCacSAP}
\begin{cases}
    \omega_a(t)&\hspace{-3mm}=\displaystyle \int_0^t e^{-\alpha(t-s)}  \int_{\R_+^2}h_a(c)\Pi^{\textup{C}}_{w(s)}(\diff x,\diff c) \diff s,\quad a{\in}\{p,d\},\\
 \displaystyle   \frac{\diff w(t)}{\diff t}&\hspace{-3mm}=M\left(\omega_p(t),\omega_d(t), w(t)\right),
\end{cases}
\end{multline}
almost surely, where, for $w{\in}K_W$, $\Pi^{\textup{C}}_{w}$ is the probability distribution defined in Proposition~\ref{def:calciumstati}.
\end{theorem}
\begin{proof}
  The application of  Theorem~\ref{th:Stoch} is straightforward.  All we have have to prove now is that ODE~\eqref{eqCacSAP} has a unique solution.

 From the representation~\eqref{RepCw}, for any $0{\le}w{\le}w'$,  the random variables $C_\infty^{w}(0)$ and $C_\infty^{w'}(0)$ can be constructed on the same probability space. The Lipschitz property of $h_a$, with the constant $L$, gives
\begin{multline*}
d_a(w,w')\steq{def}\left|\E\left[h_a(C_\infty^{w}(0))\right]{-}\E\left[h_a(C_\infty^{w'}(0))\right]\right|\le L\E\left[|C_\infty^{w}(0){-}C_\infty^{w'}(0)|\right]\\=
C_2L\E\left[\left|\int_{-\infty}^0\hspace{-3mm}  e^{\gamma s}{\cal N}_{\beta,X_\infty^w}(\diff s)
{-}\int_{-\infty}^0\hspace{-3mm}  e^{\gamma s}{\cal N}_{\beta,X_\infty^{w'}}(\diff s)\right|\right].
\end{multline*}
with~\eqref{eqInvX}, we have $X_\infty^{w}(t){=}wX_\infty^{1}(t)$ for all $t$ and, therefore,
\begin{multline*}
\frac{d_a(w,w')}{C_2L}\le
\E\left[\int_{-\infty}^0\hspace{-3mm}  e^{\gamma s}{\cal P}\left[\left(\beta(wX_\infty^1(s)),\beta(w'X_\infty^1(s))\right] ,\diff s\right]\right]\\
=\beta(w'{-}w)\E\left[\int_{-\infty}^0\hspace{-3mm}  e^{\gamma s}X_\infty^1(s)\diff s\right]=\frac{\beta}{\gamma}(w'{-}w)
\end{multline*}
Let $(w(t))$, $(w'(t))$ be  two solutions of ODE~\eqref{eqCacSAP} with the same initial point; then
\begin{align*}
\Delta_a(t) &\steq{def} \left|\int_0^t e^{-\alpha(t-s)}\left[\int_{\R_+^2}h_a(c)\Pi^{\textup{C}}_{w(s)}(\diff x,\diff c){-}  \int_{\R_+^2}h_a(c)\Pi^{\textup{C}}_{w'(s)}(\diff x,\diff c)\right]\diff s\right|\\
  &\le \int_0^t  \left|\E\left[h_a(C_\infty^{w(s)}(0))\right]{-}\E\left[h_a(C_\infty^{w'(s)}(0))\right]\right|\diff s\le C_2L\frac{\beta}{\gamma} \int_0^t |w(s){-}w'(s)|\diff s.
\end{align*}
With Relation~\eqref{eqCacSAP} and the Lipschitz property of $M$, we get, for $t{\le}T$,
\begin{multline*}
  \|w{-}w'\|_t {\steq{def}}\sup_{s{\leq}t}|w(s){-}w'(s)|\le  L_M\int_0^t e^{-\mu(t-s)}(\Delta_p(s){+}\Delta_d(s))\diff s\\
  \leq  2TL_MC_2L\frac{\beta}{\gamma}\int_0^t\|w{-}w'\|_s\diff s.
\end{multline*}
This implies that $(\|w{-}w'\|_t)$ is identically $0$.
The theorem is proved.

\end{proof}
The Lipschitz assumptions for the functions $(h_p,h_d)$ of Assumptions~C do not apply to the classical threshold functions $(S_{p},S_{d})$ from~\citet{graupner_calcium-based_2012} defined by
\begin{equation}\label{eq:CaThresh}
S_a(x)=\ind{x{\ge}\theta_a}, \quad x{\ge}0.
\end{equation}
Additionally, even in the case of  Lipschitz functions, the quantities
\[
\int_{\R{\times}\R_+}h_a(c)\Pi^{\textup{C}}_{w}(\diff x,\diff c), \;\; a{\in}\{p,d\},
\]
of the ODE~\eqref{eqCacSAP} do not have a closed form expression in general.

Theorem~\ref{th:CaStoch} highlights the importance of the calcium concentration on the dynamics of the synaptic weight. Interestingly, calcium has been the subject of a wide array of experimental studies, and biologists have developed several means to follow its concentration both locally and globally during experiments.
In particular, it is now possible to monitor calcium concentration in dendrites of postsynaptic neurons during stimulations with calcium fluorescence indicators, such as GCaMP for example~\citet{nakai_high_2001}, See~\citet{higley_calcium_2008}.
It may be therefore possible to infer these cumulative functions from such experiments and study the dynamics of Theorem~\ref{th:CaStoch} for those realistic calcium concentrations.

From the point of view of numerical analysis, it is quite difficult to obtain some simple numerical results to express solutions of the ODE~\eqref{eqCacSAP}.  It could be done, by simulations, to estimate the quantities $\E_{\Pi^{\textup{C}}_{w}}(h_a(C))$, $a{\in}\{p,d\}$ for a large number of values for $w$. A recent article (see~\citet{graupner_natural_2016}) has derived some approximations for specific cases.

For this reason, the next section investigates a class of discrete calcium-based models for which the invariant distributions have an explicit expression which can be used in practice.\section{Discrete Models of  Calcium-Based Rules}
\label{ap:stochqueu}

\begin{figure}[ht]
\centerline{\includegraphics[]{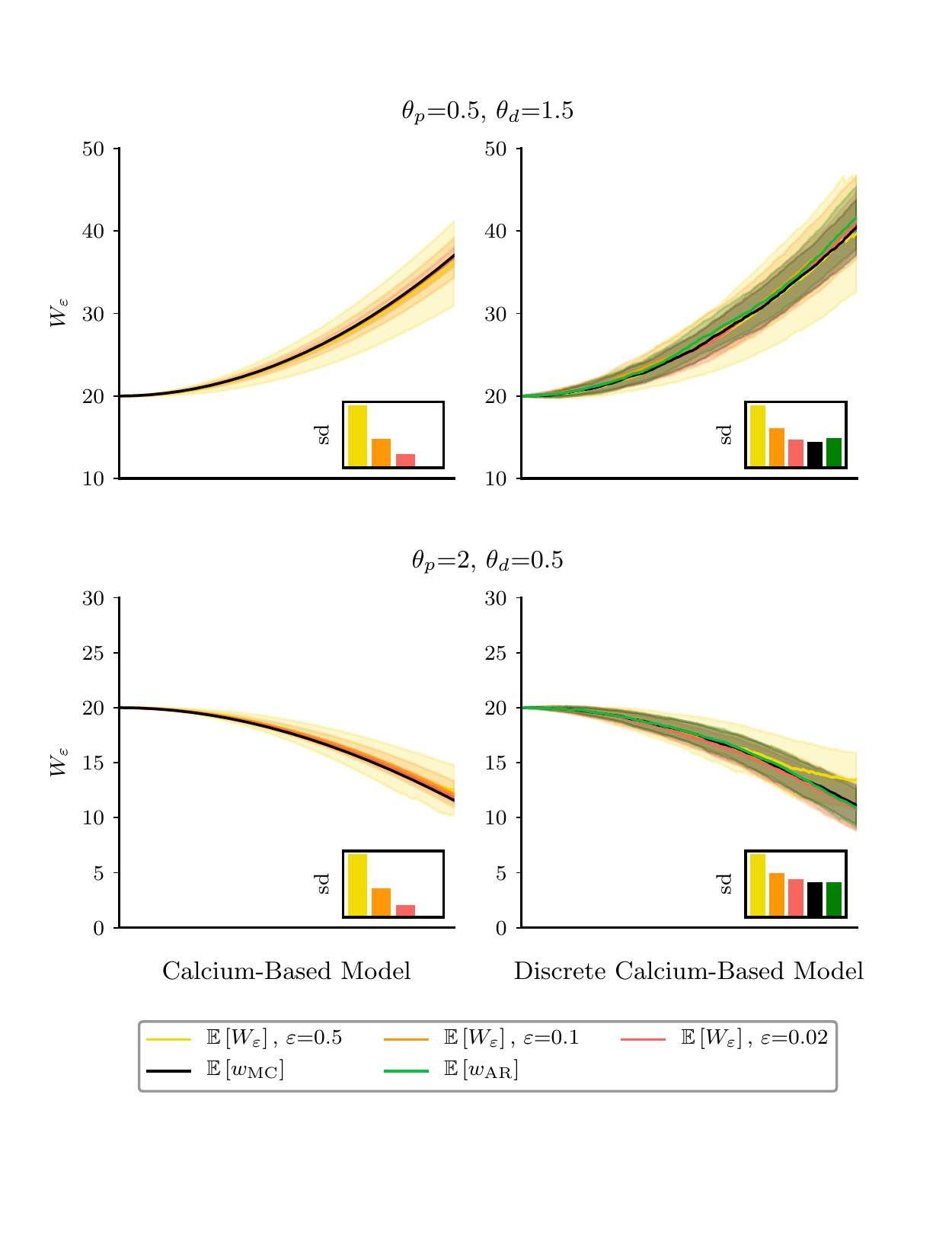}}
\caption[Comparison between continuous and discrete calcium-based models]{Comparison between Continuous and Discrete Calcium-Based Models\\
$\lambda{=}0.1$, $\gamma{=}2$, $C_1{=}C_2{=}1$, $B_p{=}2$, $B_d{=}1$, $\beta(x){=}(0.01 x)^+$, $\alpha{=}0.01$ and $\delta{=}0$.\\
For the continuous model, we took $M(\omega_p, \omega_d, w){=}\omega_p{-}\omega_d$.\\
Inset, standard deviations, sd, of $W_{\varepsilon}$ and $w$ at the end of the simulations.\\
The expected value of $w_{\textup{MC}}$ is computed with Monte Carlo estimations of $\Pi_w^{\textup{CQ}}$.\\
The expected value of  $w_{\textup{AR}}$ is computed with  estimated $\Pi_w^{\textup{CQ}}$ with the expressions of Section~\ref{app:calculCQ}.
}
\label{fig:modelcalcium}
\end{figure}

In this section, we study a simple model of plasticity where  the membrane potential~$X$, the calcium concentration~$C$, and the synaptic weight $W$ are integer-valued variables. It amounts to representing these three quantities $X$, $C$, and $W$ as multiples of a ``quantum'' instead of a continuous variable. A general class of such discrete models has been introduced in Section~4 of~\citet{robert_stochastic_2020_1}.

This amounts to describing the model of plasticity as a chemical reaction network of interacting  chemical species: ${\cal C}$ (calcium), ${\cal W}$ synaptic quanta, ${\cal X}$ ions. The associated chemical reactions could be described as
\[
\begin{cases}
  \emptyset \stackrel{\lambda}{\rightharpoonup} W{\cal X}{+}C_1{\cal C},\\
            {\cal X}  \stackrel{\beta}{\rightharpoonup} C_2{\cal C},
\end{cases}
\begin{cases}
B_d{\cal W} \rightharpoonup \emptyset  \rightharpoonup B_p{\cal W}, \\{\cal W}  \stackrel{\mu}{\rightharpoonup}\emptyset,
\end{cases}
\begin{cases}
{\cal X} \stackrel{1}{\rightharpoonup} \emptyset, \\
{\cal C} \stackrel{\gamma}{\rightharpoonup} \emptyset.
\end{cases}
\]
In this setting, the state variable is the vector  of the number of copies of the different chemical species. 
See~\citet{Feinberg} for a general introduction to chemical reaction networks and also Chapter~2 of~\citet{AndersonKurtz}.  It should be noted that our model is not strictly speaking a chemical reaction network since some reactions rates are  defined by the processes $(\Omega_a(t))$, $a{\in}\{p,d\}$.

The state of the system is associated to the solution of the following SDEs;
\[
\begin{cases}
\quad \diff X(t) &= \displaystyle{-}\sum_{i=1}^{X(t-)}\mathcal{N}_{1,i}(\diff t) +W(t{-})\mathcal{N}_{\lambda}(\diff t)-\sum_{i=1}^{X(t-)}\mathcal{N}_{\beta,i}(\diff t),\\
\quad \diff C(t) &= \displaystyle -\sum_{i=1}^{C(t-)}\mathcal{N}_{\gamma,i}(\diff t)+ C_1\mathcal{N}_{\lambda}(\diff t)+C_2\sum_{i=1}^{X(t-)}\mathcal{N}_{\beta,i}(\diff t),\\
\quad \diff \Omega_{a}(t) &=\displaystyle\quad
        \left[\rule{0mm}{4mm}{-}\alpha\Omega_{a}(t) {+} h_a(C(t))\right]\diff t,\quad a{\in}\{p,d\},\\
\quad \diff W(t) &=\displaystyle -\sum_{i=1}^{W(t-)}\mathcal{N}_{\mu,i}(\diff t) +B_p{\cal N}_{\Omega_{p}(t-)}(\diff t)-B_d\ind{W(t-){\ge}B_d}{\cal N}_{\Omega_{d}(t-)}(\diff t),
\end{cases}
\]
where  $C_1$, $C_2{\in}\N$ and, for $a{\in}\{p,d\}$, $B_a{\in}\N$ and $h_a$ is a nonnegative function.
For $\xi{>}0$, ${\cal N}_\xi$ (resp., $({\cal N}_{\xi,i})$) is a Poisson process on $\R_+$ with rate $\xi$ (resp., an i.i.d. sequence of such point processes).
For $a{\in}\{p,d\}$, as before, the notation ${\cal N}_{\Omega_{a}(t-)}(\diff t)$ stands for ${\cal P}\left[(0,\Omega_{a}(t-)),\diff t\right]$, where ${\cal P}$ is a Poisson process in $\R_+^2$ with rate $1$. All Poisson processes are assumed to be independent.

Outside the leaking mechanism, the time evolution of the discrete random variable $(W(t))$ is driven  by two inhomogeneous Poisson processes, one for potentiation and the other for depression with respective intensity functions $(\Omega_p(t))$ and $(\Omega_d(t))$.

The scaling is done in an analogous way as in Section~\ref{sec:scaling}. The corresponding SDEs are then expressed as
\begin{equation}\label{DCMS}
\begin{cases}
\quad \diff X_\eps(t) &\displaystyle\hspace{-3mm}= {-}\sum_{i=1}^{X_\eps(t-)}\mathcal{N}_{1/\eps,i}(\diff t) +W_\eps(t{-})\mathcal{N}_{\lambda/\eps}(\diff t)-\sum_{i=1}^{X_\eps(t-)}\mathcal{N}_{\beta/\eps,i}(\diff t),\\
\quad \diff C_\eps(t) &\displaystyle\hspace{-3mm}= -\sum_{i=1}^{C_\eps(t-)}\mathcal{N}_{\gamma/\eps,i}(\diff t)+C_1\mathcal{N}_{\lambda/\eps}(\diff t)+C_2\sum_{i=1}^{X_\eps(t-)}\mathcal{N}_{\beta/\eps,i}(\diff t),\\
\quad \diff \Omega_{\eps,a}(t) &\displaystyle\hspace{-3mm}=
        -\alpha\Omega_{\eps,a}(t)\diff t {+} h_a(C_\eps(t))\diff t,\quad a{\in}\{p,d\},\\
\quad \diff W_\eps(t) &\displaystyle\hspace{-3mm}={-}\hspace{-3mm}\sum_{i=1}^{W_\eps(t-)}\hspace{-2mm}\mathcal{N}_{\mu,i}(\diff t) {+}B_p{\cal N}_{\Omega_{\eps,p}(t)}(\diff t){-}B_d\ind{W_\eps(t-){\ge}B_d}{\cal N}_{\Omega_{\eps,d}(t)}(\diff t),
\end{cases}
\end{equation}

\begin{definition}[Fast Processes]\label{FVD}
For a fixed $W{=}w$, the fast variables of the SDEs~\eqref{DCMS} are associated to  a Markov process  $(X^w(t),C^w(t))$ on $\N^2$  whose transition rates are given by, for $(x,c){\in}\N^2$,
\[
(x,c)\longrightarrow
\begin{cases}
\hspace{1mm}(x{+}w,c{+}C_1) & \lambda, \\
\hspace{1mm}(x{-}1,c)   & x,
\end{cases}
\hspace{2cm}
\longrightarrow
\begin{cases}
\hspace{1mm}(x,c{-}1)   & \gamma c, \\
\hspace{1mm}(x{-}1,c{+}C_2)  & \beta x.
\end{cases}
\]
\end{definition}
The next result is the equivalent of Theorem~\ref{th:CaStoch} in a discrete setting.

\begin{theorem}[Averaging Principle for Discrete Calcium-Based Model] \label{th:CaStochD}
If $h_p$ and $h_d$ are functions on $\N$ with a finite range of values, as $\eps$ goes to $0$, the family of processes $(\Omega_{\eps,p}(t),\Omega_{\eps,d}(t),W_{\eps}(t))$ defined by Relations~\eqref{DCMS} converges in distribution to the unique solution $(\omega_p(t),\omega_d(t),w(t))$ of the relations
\begin{equation}\label{eqCacSAPD}
\begin{cases}
\omega_a(t)&\hspace{-3mm}=\displaystyle \int_0^t e^{-\alpha(t-s)}\int_{\N^2} h_{a}(c) \Pi^{\textup{CQ}}_{w(s)}(\diff x,\diff c)\,\diff s,\quad a{\in}\{p,d\},\\
 \diff w(t) &\hspace{-3mm}= \displaystyle-\sum_{i=1}^{w(t-)}\mathcal{N}_{\mu,i}(\diff t) {+}B_p{\cal N}_{\omega_{p}(t)}(\diff t){-}B_d\ind{w(t{-}){\ge}B_d}{\cal N}_{\omega_{d}(t)}(\diff t),
\end{cases}
\end{equation}
where $\Pi^{\textup{CQ}}_{w}$ is the invariant distribution of the Markov process of Definition~\ref{FVD}.
\end{theorem}
The threshold functions $(S_p,S_d)$ defined by Relations~\eqref{eq:CaThresh} and used in classical models of calcium-based STDP~\citet{graupner_calcium-based_2012} satisfy the conditions of the theorem. For the proof, see the appendix of~\citet{robert_stochastic_2020}.

The theorem shows that the limiting process $(w(t))$ is a  jump process on $\N$ driven by  two nonhomogeneous Poisson processes whose intensity functions $(\omega_a(t))$, $a{\in}\{p,d\}$ are continuous.

The explicit expression of the invariant distribution of $(C^w(t))$ is given in Proposition~4.3 of~\citet{robert_stochastic_2020_1}. Only the distribution of the calcium variable $C^w$ is considered due to its role in  the expression of $(\omega_a(t))$, $a{\in}\{p,d\}$ in Theorem~\ref{th:CaStochD}.

\begin{proposition}[Equilibrium of Fast Process] \label{EFPD}
For $w{\in}\N$, the Markov process on $\N^2$ of Definition~\ref{FVD} has a unique invariant distribution $\Pi^{\textup{CQ}}_w$, and the generating function of $C^w$ is given by, for $u{\in}[0,1]$,
\begin{equation}\label{GenC}
    E\left(u^{C^w}\right)=
    \exp\left({-}\lambda\int_{0}^{+\infty}\left(1{-}\Delta(u,s,w)\right)\diff s\right)
\end{equation}
with
\[
\Delta (u,s,w) = \left(\rule{0mm}{4mm}1{+}(u{-}1)p_1(s)\right)^{C_1}\left(1{+}\sum_{i=1}^{C_2}(u{-}1)^k p_2(s,k)\right)^w
\]
\[
p_1(s)=e^{-\gamma s}\text{ and }
  p_2(s,k) = \frac{\beta}{\beta{+}1{-}\gamma k} \binom{C_2}{k}\left(e^{-\gamma k s}{-}e^{-(\beta+1)s}\right).
\]
\end{proposition}
We present in Figure~\ref{fig:modelcalcium} simulations for different values of $\theta_p$ and $\theta_d$ of the
continuous model (Section~\ref{cbmsec}) with step functions $S_{a}$ (left) and of the discrete model (right).
In particular, we simulate the scaled system for different values of $\eps$ and we also estimate the solution of~\eqref{eqCacSAP}
and~\eqref{eqCacSAPD}, $w_{\textup{MC}}$, using Monte Carlo estimations to compute $\Pi^{\textup{C}/\textup{CQ}}_w(C{\geq}\theta_a)$.

Moreover, for the discrete case, we are able to compute $\Pi_w^{\textup{CQ}}(C{\geq}n)$ for $n{=}0$, $1$, $2$; see Section~\ref{app:calculCQ}.  Based on these analytical results, we are able to obtain the numerical values of the parameters of the dynamic of the asymptotic process $(w_{\textup{AR}}(t))$.
Simulations of the expected values of $(w_{\textup{AR}}(t))$ are represented in green in Figure~\ref{fig:modelcalcium}.

These simulations illustrate Theorems~\ref{th:CaStoch} and~\ref{th:CaStochD}, with the convergence of the scaled processes $W_{\eps}$ towards our asymptotic process. For the continuous case, we observe that even if the step function $S_{a}$ does not verify the conditions of~\ref{th:CaStoch}, convergence seems to hold anyway.
This is also illustrated by the decrease in standard deviations (inset) as $\eps$ goes to $0$. For the discrete case, we note the same phenomenon for the expected value and the standard deviation. Recall that the limiting process is stochastic in this context. Finally, it also shows that, qualitatively, the two classes of models continuous/discrete behave quite similarly.

\printbibliography

\appendix

\label{ap:comparison}
\section{Averaging Principles for Models Without Exponential Filtering}
\label{apap:noexp}
In Section~SM2 of~\citet{robert_stochastic_2020_1} more ``direct'' dynamics for the time evolution of synaptic weight have been presented. For $a{\in}\{p,d\}$, the update at time $t$  depends only on the instantaneous synaptic plastic processes $$\Gamma_a(\mathcal{N}_\lambda,\mathcal{N}_{\beta,\overline{X}})(\diff t) = n_{a,0}(Z(t-))\diff t{+}n_{a,1}(Z(t-)){\cal N}_\lambda(\diff t){+}n_{a,2}(Z(t-)){\cal N}_{\beta,X}(\diff t)$$ at time $t$. The corresponding synaptic weight process $(\overline{W}(t))$ satisfies the relation
\[
  \diff \overline{W}(t) = \displaystyle \overline{M}\left(\Gamma_p(\mathcal{N}_\lambda,\mathcal{N}_{\beta,\overline{X}}), \Gamma_d(\mathcal{N}_\lambda,\mathcal{N}_{\beta,\overline{X}}), \overline{W}(t)\right)(\diff t),
    \]
for some functional $\overline{M}$.

Recall that for our model,  the dynamic of the synaptic weight $(W(t))$ is defined by,
\[
 \diff W(t) = M\left(\Omega_p(t),\Omega_d(t), W(t)\right)\diff t,
\]
where $(\Omega_a(t))$, $a{\in}\{p,d\}$, is a filtered/smoothed version of $\Gamma_a(\mathcal{N}_\lambda,\mathcal{N}_{\beta,\overline{X}})$,
\[
\diff \Omega_a(t)\displaystyle ={-}\alpha\Omega_a(t)\diff t {+}n_{a,0}(Z(t))\diff t+
    \Gamma_a(\mathcal{N}_\lambda,\mathcal{N}_{\beta,X})(\diff t)
\]

It turns out that a stochastic averaging principles also holds for the model without an exponential filtering. We first introduced the scaled version of this system.
\begin{definition}[Scaled Dynamical System for Instantaneous Plasticity]
\label{def:procnoexp}
We define the stochastic process $(\overline{X}(t),\overline{W}(t))$ with initial state $(x_0,w_0)$, satisfying the evolution equations, for $t{>}0$,
\begin{equation}\label{Systnoexp}
  \begin{cases}
\quad  \diff \overline{X}_\eps(t) = \displaystyle {-}\frac{1}{\eps}\overline{X}_\eps(t)\diff t+\overline{W_\eps}(t{-})\mathcal{N}_{\lambda/\eps}(\diff t)-g(\overline{X}_\eps(t-))\mathcal{N}_{\beta/\eps,\overline{X}_\eps}(\diff t),\\
\quad  \diff \overline{Z}_\eps(t)  = \displaystyle \frac{1}{\eps}\left(\rule{0mm}{5mm}{-}\gamma\odot \overline{Z}_\eps(t){+}k_0\right)\diff t\\
    \hspace{1cm}{+}k_1(\overline{Z}_\eps(t{-}))\mathcal{N}_{\lambda/\eps}(\diff t){+}k_2(\overline{Z}_\eps(t{-}))\mathcal{N}_{\beta/\eps,\overline{X}_\eps}(\diff t),\vspace{0.5em}\\
\quad   \diff \overline{W}_\eps(t) = \displaystyle \eps \overline{M}\left(\Gamma_p(\mathcal{N}_\lambda,\mathcal{N}_{\beta,\overline{X}}),\Gamma_d(\mathcal{N}_\lambda,\mathcal{N}_{\beta,\overline{X}}), \overline{W}(t)\right)(\diff t), \\
  \end{cases}
\end{equation}
where $\Gamma_p$ and $\Gamma_d$ are plasticity kernels.
The functional $\overline{M}$ is defined by
\begin{align}\label{eqMB}
\overline{M}{:}\quad &{\cal M}_+(\R_+)^2{\times}\R \mapsto {\cal M}_+(\R_+)\\
                  &\left(\Gamma_p,\Gamma_d, w\right) \rightarrow \overline{M}(\Gamma_p,\Gamma_d, w).\notag
\end{align}
\end{definition}

We have to modify Assumptions~B-(d) by Assumptions~B*-(d), in the following way,
$\overline{M}$ can be decomposed as, $\overline{M}(\Gamma_p,\Gamma_d,w){=}\overline{M}_p(w)\Gamma_p{-}\overline{M}_d(w)\Gamma_d - \mu w$, where $\overline{M}_{a}(w)$ is non-negative continuous function, and,
\[
   \overline{M}_a(w)\leq C_M,
\]
for all $w{\in}K_W$, for $a{\in}\{p,d\}$.

An analogue of Theorem~\ref{th:Stoch} in this context is the following result.
\begin{theorem}[Averaging Principle for Instantaneous Plasticity]
\label{th:StochNoExp}
Under Assumptions~A and~B* and for initial conditions satisfying Relation~\eqref{InCond}, there exists $S_0{\in}(0,{+}\infty]$, such that  the family of processes $(\overline{W}_{\eps}(t),t{<}S_0)$ associated to Relations~\eqref{Systnoexp} and~\eqref{eqMB}, is tight for the convergence in distribution as $\eps$ goes to $0$.  Almost surely, any  limiting point $(\overline{w}(t),t{<}S_0)$ satisfies the relation
\begin{equation}\label{eqAPIP}
\diff \overline{w}(t)=\int_{\R{\times}\R_+^\ell}\hspace{-4mm} \overline{M}\left(\rule{0mm}{4mm}\left[\rule{0mm}{3mm}(n_{a,0}(z){+}\lambda n_{a,1}(z) {+} \beta(x)n_{a,2}(z))\diff t\right]_{a{\in}\{p,d\}}, \overline{w}(t)\right)\Pi_{\overline{w}(t)}(\diff x,\diff z).
\end{equation}
where, for $w{\in}K_W$, $\Pi_w$ is the invariant measure $\Pi_w$ of the operator $B_w^F$ of Relation~\eqref{eq:bfast}.
\end{theorem}
\begin{proof}
Due to the specific expression of $\overline{M}$, the arguments follow  the ones used in~\citet{robert_stochastic_2020}. The proof is skipped. 
\end{proof}

\subsection*{Comparison with Theorem~\ref{th:Stoch}}
Both theorems show that the dynamics of the synaptic weight $w$ in the decoupled stochastic system depend on an integral over the stationary distribution of the fast process. However, in Theorem~\ref{th:Stoch}, the averaging property occurs at the level of the synaptic plasticity processes $\omega_{a}$,
\[
 \frac{\diff \omega_a(t)}{\diff t}={-}\alpha \omega_a(t)+\int_{\R{\times}\R_+^\ell} \left[\rule{0mm}{4mm}n_{a,0}(z){+}\lambda n_{a,1}(z) {+} \beta(x)n_{a,2}(z)\right]\diff\Pi_{w(t)}(x,z),
\]
and, the function $M$ is applied afterwards to have the update of the synaptic weight $w$,
\[
    \frac{\diff w(t)}{\diff t}=M(\omega_p(t),\omega_d(t),w(t)).
\]
In Theorem~\ref{th:StochNoExp}, with no exponential filtering, the averaging is applied directly at the level of the synaptic update,
\begin{align*}
\diff \overline{w}(t) &=\int_{\R{\times}\R_+^\ell}\hspace{-4mm} \overline{M}\left(\rule{0mm}{4mm}\left[\rule{0mm}{3mm}(n_{a,0}(z){+}\lambda n_{a,1}(z) {+} \beta(x)n_{a,2}(z))\diff t\right]_{a{\in}\{p,d\}}, \overline{w}(t)\right)\Pi_{\overline{w}(t)}(\diff x,\diff z).
\end{align*}

In particular, with a linear function $M$, both models are equivalent except for the exponential filtering of the plasticity kernels.

\section{Links with Models of Physics: A Heuristic Approach}\label{app:comparisoncompheur}
In this section we give a, non-rigorous,  derivation of Relation~\eqref{eqKemp} of~\citet{kempter_hebbian_1999}  to establish a connection with our main results in this specific case. For $w{\in}K_W$, from the definition of $\Phi_a$, $a{\in}\{p,d\}$,
\begin{align*}
  \int_{-\infty}^0&\Phi_a(s)\mathbb{E}_{\Pi^{\textup{PA}}_w}\left(\frac{\mathcal{N}_{\beta,X}[0, h]}{h}\frac{\mathcal{N}_{\lambda}[s, s {+} h]}{h}\right)\diff s \\
     =\ &B_{a,1}\mathbb{E}_{\Pi^{\textup{PA}}_w}\left(\int_{-\infty}^{0}\exp(\gamma_{a,1}s)\mathbb{E}_{\Pi^{\textup{PA}}_w}\left[\frac{\mathcal{N}_{\beta,X}[0, h]}{h}\frac{\mathcal{N}_{\lambda}[s, s + h]}{h}\middle| \mathcal{F}_0\right]\diff s\right)\\
      =\ &B_{a,1}\mathbb{E}_{\Pi^{\textup{PA}}_w}\left(\mathbb{E}_{\Pi^{\textup{PA}}_w}\left[\frac{\mathcal{N}_{\beta,X}[0, h]}{h}\middle| \mathcal{F}_0\right]\int_{-\infty}^{-h}\exp(\gamma_{a,1} s)\frac{\mathcal{N}_{\lambda}[s, s + h]}{h}\diff s\right)\\
  \sim\ & B_{a,1}\mathbb{E}_{\Pi^{\textup{PA}}_w}\left( \beta(X(0))\int_{-\infty}^{-h}\exp(\gamma_{a,1} s)\frac{\mathcal{N}_{\lambda}[s, s + h]}{h}\diff s\right)
  \shortintertext{and, if ${\cal N}_\lambda{=}(t_n, n{\in}\Z)$, with $t_0{\le}0{<}t_1$,}  
  =\ &B_{a,1}\mathbb{E}_{\Pi^{\textup{PA}}_w}\left(\beta(X(0))\sum_{t_n{\le}-h}\frac{1}{h}\int_{t_n}^{t_n+h}\hspace{-5mm}\exp(\gamma_{a,1} s)\diff s\right) \\
  {\sim}\ &B_{a,1}\mathbb{E}_{\Pi^{\textup{PA}}_w}\left(\beta(X(0))\sum_{n\le 0}\exp(\gamma_{a,1} t_n)\right)\\
         =\ &\mathbb{E}_{\Pi^{\textup{PA}}_w}\left(\beta(X(0))Z_{a,1}(0)\right) = \int_{\R_+^5}\beta(x)z_{a,1}\Pi^{\textup{PA}}_{w}(\diff x,\diff z)
\end{align*}
by using a representation of $Z_{a,1}(0)$ similar to that of $X_\infty^w(0)$ with Relation~\eqref{eqInvX}. Similarly,
\[
\int_0^{+\infty}\Phi_a(s)\mathbb{E}_{\Pi^{\textup{PA}}_w}\left(\frac{\mathcal{N}_{\beta,X}[0, h]}{h}\frac{\mathcal{N}_{\lambda}[s, s {+} h]}{h}\right)\diff s
\sim \int_{\R_+^5}\lambda z_{a,2}\Pi^{\textup{PA}}_{w}(\diff x,\diff z).
\]

\subsection*{Extensions}
The interest of  Relation~\eqref{eqKemp} is that it may be formulated  for a general plasticity curve $\Phi_{a}$ for all-to-all pair-based models.
Recall that the corresponding plasticity kernels are of class~${\cal M}$ only for exponential functions. We conjecture that under the conditions, for $a{\in}\{p,d\}$,
\begin{itemize}
    \item $\displaystyle\int_{-\infty}^{+\infty}\left |\Phi_a(s) \right| \diff s < + \infty$;
    \item $\displaystyle\lim_{t\rightarrow 0+}\Phi_a(t)$ and $\displaystyle\lim_{t\rightarrow 0-}\Phi_a(t)$ exist,
\end{itemize}
the convergence of the scaled process to the ODE~\eqref{eqKemp} with a convenient $\widetilde{\mu}$ should hold.
For Markovian plasticity kernels, this is done by using  Markov properties of the fast processes $(X_\eps(t),Z_\eps(t))$. See~\citet{robert_stochastic_2020}.
We  do not have this tool in the case of a general plasticity curve.
The proof of such an extension should require an additional analysis.

\section{Computation of $\Pi_w^{\textup{CQ}}$ for $C_1{=}C_2{=}1$}
\label{app:calculCQ}

\begin{proposition}[Equilibrium of Fast Process] \label{EFPD1}
For  $C_1{=}C_2{=}1$ and $w{\in}\N$, the Markov process on $\N^2$ of Definition~\ref{FVD} has a unique invariant distribution $\Pi_w$, and  if the distribution of $(X_w,C_w)$ is $\Pi_w$, the generating function of $C_w$ is given by, for $u{\in}[0,1]$,
\begin{equation}\label{GenC1}
    g_w(u)=E\left(u^{C_w}\right)=
    \exp\left({-}\lambda\int_{0}^{+\infty}\left(1{-}\left(1{-}e^{-\gamma s}{+}ue^{-\gamma s}\right)\left(1{-}p(s)+u p(s)\right)^w\right)\diff s\right),
\end{equation}
with
\[
p(s)=\left.\left(e^{{-}\gamma s}{-}e^{{-}(\beta{+}1)s}\right)\right/(\beta{+}1{-}\gamma).
\]
\end{proposition}

In particular, knowing that,
\[
    \Pi^{\textup{CQ}}_w(C{\geq}0)=1,
\]

we can easily compute,
\[
    \Pi^{\textup{CQ}}_w(C{\geq}1)=1 - \Pi^{\textup{CQ}}_w(C{=}1)
\]
with,
\[
   g_w(0)=\exp\left({-}\lambda\int_{0}^{+\infty}\left(1{-}\left(1{-}e^{-\gamma s}\right)\left(1{-}p(s)\right)^w\right)\diff s\right)
\]
and,
\[
    \Pi^{\textup{CQ}}_w(C{\geq}2)=1 - g_w(0) - g'_w(0)
\]
with,
\[
    g'_w(0) = \lambda\left[\int_{0}^{+\infty}\left(e^{-\gamma s}\left(1{-}p(s)\right)^w
        + wp(s)\left(1{-}e^{-\gamma s}\right)\left(1{-}p(s)\right)^{w-1}\right)\diff s \right] g_w(0).
\]

\end{document}